\documentclass[11pt,a4paper]{article}

\usepackage[english]{babel}
\usepackage{amsmath}
\usepackage{amsfonts,amsthm,amssymb}
\usepackage{blkarray}
\usepackage{latexsym,bm}
\usepackage{indentfirst}
\usepackage{graphicx}
\usepackage{color}
\usepackage{tcolorbox}
\usepackage{tikz}
\usepackage{caption}
\usepackage{multicol}
\usepackage{multirow}
\usepackage{appendix}
\usepackage{makecell}
\usepackage[letterpaper,top=2cm,bottom=2cm,left=3cm,right=3cm,marginparwidth=1.75cm]{geometry}
\usepackage[colorlinks=true,anchorcolor=blue,filecolor=blue,linkcolor=blue,urlcolor=blue,citecolor=red]{hyperref}

\newtheorem{theorem}{Theorem}[section]

\newtheorem{proposition}{Proposition}
\newtheorem{lemma}{Lemma}[section]
\newtheorem{false statement}{False statement}
\newtheorem{corollary}{Corollary}[section]

\theoremstyle{definition}

\newtheorem{claim}{Claim}

\numberwithin{case}{subsection}

\title{\textbf{Upper bounds of the second largest eigenvalue of graphs}
\thanks{
This work is supported by National Natural Science Foundation of China (Nos. 12171154, 12301438), Chenguang Program of Shanghai Education Development Foundation and Shanghai Municipal Education Commission (No. 23CGA37), and China Scholarship Council (No. 202506740024).}
}
\author{
	Zhiwen Wang,\,\, Zihao Geng,\,\,
	Ji-Ming Guo\thanks{Corresponding author.\\ 
	\null\hspace{6.3mm}Email addresses:  walkerwzw@163.com (Z. Wang); gzh635516986@163.com (Z. Geng); jimingguo@hotmail.com (J.-M. Guo).}\vspace{4mm}\\
	School of Mathematics, East China University of Science and Technology,\\ Shanghai, 200237, P. R. China.
}
\date{\null}
\begin{document}
\maketitle

\begin{abstract}

Let $\lambda_i(G)$ denote the $i$-th largest eigenvalue of adjacency matrix of a graph $G$.
Gerschgorin’s Theorem indicates $\lambda_1(G)$ belongs to the largest disk, i.e., $\lambda_1(G)\le\Delta_1(G)$, where $\Delta_i(G)$ is the $i$-th largest degree of $G$.
We show that $\lambda_2(G)$ lies in the second largest disk. That is, in detail,
$$\lambda_2(G)<\Delta_2(G)-\frac{1}{n^2}.$$

A classical theorem proved by Hong [\textit{Linear Algebra Appl.} 1988] states that $\lambda_1(G)\le\sqrt{2m-n+1}$ for a connected graph $G$ with $n$ vertices and $m$ edges,
where the equality holds if and only if $G$ is a star $S_n$ or a complete graph $K_n$.
We give a refinement of Hong's theorem by showing $$\lambda_1(G)<\sqrt{2m-n}$$ for any connected graph $G\not\in\left\{S_n,S^1_{n-1},K_n,K^1_{n-1}\right\}$.
Based on this improved upper bound of $\lambda_1(G)$, for a connected graph $G$ with $n$ vertices and $m$ edges, we are able to prove a sharp upper bound of $\lambda_2(G)$ that
$$\lambda_2(G)\le\sqrt{m-\frac{n}{2}-\frac{1}{2}},$$
except $G$ is obtained from two disjoint $S_\frac{n}{2}$ by adding an edge between a pendant vertex of each star.
Moreover, we provide a complete characterization to extremal graphs attaining the equality.
\end{abstract}

\section{Introduction}\label{section:1}

There is no question that a strong correlation exists between eigenvalues and the capacity to unveil the countless mysteries inherent in graphs.
For a graph $G$ of order $n$, the eigenvalues of $G$ are eigenvalues of its adjacency matrix, denoted in order by $\lambda_1(G)\ge \lambda_2(G)\ge\cdots\ge \lambda_n(G)$. 
The historic study on eigenvalues focus on the largest eigenvalue $\lambda_1$, which is also known as the spectral radius.
The story may begin with Gerschgorin's Theorem \cite{B65}, which indicates that all $n$ disks share a common center and the largest eigenvalue belongs to the largest disk. This means $$\lambda_1(G)\le \Delta_1(G),$$ where $\Delta_i(G)$ denotes the $i$-th largest degree of $G$.
Recently, using the crucial bound $\lambda_1(H)\le \Delta_1(H)$ for signed graph $H$, Huang \cite{H19} solved the outstanding Sensitivity Conjecture in theoretical computer science, proposed by Nisan and Szegedy \cite{NS94}.

As for the second largest eigenvalue, there are several known results related to graph properties.
For instance, a graph with large gap between the largest and second largest eigenvalues has good connectivity, expansion and randomness properties.
Mention a $\Delta$-regular graph has the largest eigenvalue $\Delta$.
Many scholars were interested in the gap between the degree $\Delta$ and the second largest eigenvalue $\lambda_2(G)$ of a $\Delta$-regular connected graph $G$.
A remarkable work by Alon and Boppana \cite{A86} showed that
\begin{align}\label{eq:byAB}
    \lambda_2(G)\ge 2\sqrt{\Delta-1}\left(1-O\Big(\frac{\log (\Delta-1)}{\log n}\Big)\right)
\end{align}
for every $\Delta$-regular graph.
They also conjectured that $\lambda_2(G)\le2\sqrt{\Delta-1}+o(1)$ when $n$ is sufficiency large.
Friedman, Kahn and Szemeredi \cite{FKS89} contributed to the conjecture, proving that $\lambda_2(G)\le 2\sqrt{\Delta-1}+\log\Delta+o(1)$.
A Ramanujan graph, introduced in \cite{LPS88}, is a connected $\Delta$-regular graph with $\max\{|\lambda_i(G)|:i\ge 2\}\le 2\sqrt{\Delta
-1}$.
In \cite{F04}, Numerical experiments indicate that random $\Delta$-regular graphs satisfy $\lambda_2<2\sqrt{\Delta-1}$.

Neglecting regularity, in terms of the maximum degree $\Delta_1$ and diameter $D$ of a graph $G$ in which there are two edges of distance at least $2k+2$, Nilli \cite{N91} in 1991 proved 
$$\lambda_2(G)\ge 2\sqrt{\Delta_1-1}-\frac{2\sqrt{\Delta_1-1}-1}{k+1},$$
improving Eq.\,(\ref{eq:byAB}) since $D\ge \frac{\log n}{\log(\Delta_1-1)}-O(1)$.

Solving a longstanding problem on equiangular lines by Jiang, Tidor, Yao, Zhang and Zhao \cite{JTYZZ21}, a key ingredient is showing an effective upper bound of multiplicity of the second largest eigenvalue. 
Then Balla \cite{B25} explored significant results on equiangular lines with assumption on the spectral gap $\Delta-\lambda_2(G)<\frac{n}{2}$ for $\Delta$-regular graph $G$.
The second largest eigenvalue is also used to classify graphs.
Originated from a problem of Hoffman (see \cite{H82}) that characterizing graphs with the second largest eigenvalue not greater than $1$, this topic received much attention.
One may refer to early paper \cite{AM85} and recent paper \cite{LCS24} and their references for related interesting works.
Although the second largest eigenvalue has strong practical applicability in many areas, its study remains highly challenging.

In this paper, we pay attention to upper bounds of the second largest eigenvalue of general graphs.
Motivated by aforementioned works, one of our results gives an upper bound for $\lambda_2(G)$ in terms of the second largest degree $\Delta_2(G)$.
We show that for any connected graph $G$ of order $n$,
$$\lambda_2(G)<\Delta_2(G)-\frac{1}{n^2}.$$
This infers $\lambda_2(G)$ lies in the second largest disk in Gerschgorin’s Theorem.
The method relies on the gap $\Delta_1(G)-\lambda_1(G)$ of 
irregular graph $G$ obtained by Cioab\u{a}, Gregory, Nikiforov \cite{CGN07}.

Back to the story of the largest eigenvalue of a graph, other than degrees, the order and the size are fundamental parameters of bounding the largest eigenvalue. 
In \cite{BH85}, Brualdi and Hoffman showed $\lambda_1(G)\le k-1$ for a graph with $\frac{k(k-1)}{2}$ edges.
Subsequently, Stanley \cite{S87} gave a generalization and showed $\lambda_1(G)\le\frac{-1+\sqrt{1+8m}}{2}$.
Hong \cite{H88} further obtained a classical theorem that for a connected graph $G$ of order $n$ and size $m$,
$$\lambda_1(G)\le\sqrt{2m-n+1},$$
with equality holds if and only if $G$ is a star or a complete graph.
A follow-up generalization to Hong's theorem may refer to \cite{GWL19, HSF01}.

Our second result gives a direct refinement on Hong's theorem.
Denote by $\mathcal{H}_n$ the family of graphs consisting of $S_n$, $S^1_{n-1}$, $K_n$ and $K^1_{n-1}$, where $S^1_{n-1}$ (or $K^1_{n-1}$, resp.) be the graph obtained by identifying a vertex of $P_2$ with a pendant vertex of a star $S_{n-1}$ (or a vertex of a complete graph $K_{n-1}$, resp.).
We prove, for a connected graph $G$ of order $n\ge 10$ and size $m$,
\begin{align}\label{eq:refined bound}
    \lambda_1(G)<\sqrt{2m-n},
\end{align}
unless $G\in\mathcal{H}_n$.
Moreover, from Cauchy interlacing theorem, it is not hard to find each graph $G$ in $\mathcal{H}_n$ satisties $\lambda_1(G)>\sqrt{2m-n}$.

In contrast to the largest eigenvalue, little literature seems to be know about bounds on other eigenvalues of graphs.
One may refer to \cite{BD84} for significant results, and \cite{AH10} for lots of open conjectures.
Our third result is a version of Hong's theorem for the second largest eigenvalue.
Based on the bound in Eq.\,(\ref{eq:refined bound}), and the relation of elements of eigenvector and graph properties, we can deduce a sharp upper bound of $\lambda_2(G)$ for a connected graph $G$, proving that
$$\lambda_2(G)\le\sqrt{m-\frac{n}{2}-\frac{1}{2}},$$
except $G$ is isomorphic to $S_{\frac{n}{2}}uvS_{\frac{n}{2}}$, a graph obtained from two copies of $S_{\frac{n}{2}}$ by connecting $uv$ such that $u$ (and $v$) is one pendant vertex of each $S_{\frac{n}{2}}$.
Furthermore, we completely characterize the extremal graphs attaining the equality.
They belongs to three families of graphs $\mathcal{J}(K_{\frac{n-1}{2}},K_{\frac{n-1}{2}})$, $\mathcal{J}(S_{\frac{n-1}{2}},S_{\frac{n-1}{2}})$ and $\mathcal{J}(K_{r+1},S_{r^2+1})$, where $r$ is an integer and $r^2+r+3=n$ (the definition is placed in Section \ref{section:4}).

\section{Degrees and eigenvalues}

Here we first introduce some terminologies.

Throughout this paper, a graph $G=(V(G),E(G))$ is considered to be simple and undirected graph with vertex set $V(G)$ and edge set $E(G)$.
The order $n(G)$ equals to $|V(G)|$, and the size $m(G)$ equals to $|E(G)|$.
If $U$ is subset of $V(G)$, then $G-U$ is the graph by removing all vertices in $U$ and its incident edges.
In particular, when $U$ has a single vertex $u$, we write $G-U$ as $G-u$ simply.

Suppose that $U$ and $W$ are two subsets of $V(G)$ with $U\cap W=\varnothing$.
Let $E(U,W)$ be the set of edges such that one end point belongs to $U$ and another one belongs to $W$, and $e(U,W)$ be the size of $E(U,W)$.
The graph $G-E(U,W)$ is the resulting graph after deleting all edges in $E(U,W)$; in particular, when $E(U,W)$ has a single edge $uw$, write $G-E(U,W)$ as $G-uw$.
Let $N(v)$ be the neighborhood of a vertex $v$, and $d(v)=|N(v)|$ be the degree of $v$.
If $H$ is a subgraph of $G$, then denote by $d_H(v)$ the number of neighbors of $v$ in $V(H)$.

As mentioned in Section \ref{section:1}, the $i$-th 
largest eigenvalue $\lambda_i(G)$ of a graph $G$ is the $i$-th 
largest eigenvalue of its adjacency matrix $A(G)$.
The largest eigenvalue $\lambda_1(G)$ of $G$ is also written as $\rho(G)$.
The degree sequence is denoted in order as $\Delta_1(G)\ge \Delta_2(G)\ge \cdots\ge \Delta_n(G)$.

The symbol $G$ is omitted when there is no ambiguity.

Gerschgorin's Theorem \cite{B65} shows that all eigenvalues $\lambda$ of a matrix $M=(M_{ij})_{n\times n}$ belong to the union of disks $|\lambda-M_{ii}|\le\sum_{1\le j\le n,j\not=i}{|M_{ij}|}$, $1\le i\le n$.
If $M=A(G)$, then we know $\lambda_1(G)\le \Delta_1(G)$.
That means the spectral radius of belongs to the largest disk.
This conclusion is well known and can be found in many monographs (see \cite{BH12} for example).

\begin{lemma}\label{lem:rho le max-degree}
    Let $G$ be a connected graph, then $$\rho(G)\le \Delta_1(G),$$ and the equality holds if and only if $G$ is a regular graph.
\end{lemma}

For an irregular graph $G$, Lemma \ref{lem:rho le max-degree} implies its spectral radius is less than the maximum degree.
In decades, the difference between $\Delta_1(G)$ and $\rho(G)$ is a topic of great interest.
A seminal work by Stevanovi\'c \cite{S04} proved that if $G$ is an irregular connected graph of order $n$ and maximum degree $\Delta_1$, then
$$\Delta_1(G)-\rho(G)>\frac{1}{2n(n\Delta_1-1)\Delta_1^2}.$$
Then Zhang \cite{Z05} improved the bound in terms of minimum degree $\Delta_n$ and diameter $D$, showing
$$\Delta_1(G)-\rho(G)>\frac{(\sqrt{\Delta_1}-\sqrt{\Delta_n})^2}{nD\Delta_1}.$$

Refining a bound $\Delta_1(G)+\lambda_n(G)>\frac{1}{n(D+1)}$ of Alon and Sudakov \cite{AS00}, 
a stronger inequality was established by Cioab\u{a}, Gregory and Nikiforov \cite{CGN07}.
Their contribution is the crucial ingredient to get the bound for $\lambda_2(G)$ in Theorem \ref{thm:2th<d2} below.

\begin{lemma}\cite{CGN07}\label{lem:irregular gap}
    If $G$ is a connected irregular graph of order $n$ and diameter $D$, then 
    $$\Delta_1(G)-\rho(G)>\frac{1}{n(D+1)}.$$
\end{lemma}

\begin{lemma}\cite{WG26+}\label{lem:null cut-set}
Let $\textbf{x}$ be an eigenvector of a graph $G$ corresponding to $\lambda_2(G)$.
Let $V_0$ be a vertex cut and $\textbf{x}\big|_{ {V_0}}=\emph{\textbf{0}}$. 
If $H$ is a component of $G-V_0$ and $\textbf{x}\big|_{ {V(H)}}\not=\emph{\textbf{0}}$,
then $\lambda_2(G)$ is an eigenvalue of $H$.
In particular, if $\textbf{x}\big|_{V(H)}>\emph{\textbf{0}}$ or $\textbf{x}\big|_{V(H)}<\emph{\textbf{0}}$, then $\lambda_2(G)$ is the largest eigenvalue of $H$.
\end{lemma}

\begin{lemma}\cite{WG26+}\label{lem:connected by negative edges}
Let $\textbf{x}$ be an eigenvector of a graph $G$ corresponding to $\lambda_2(G)$ and $H$ be an induced subgraph of $G$. 
If $\textbf{x}\big|_{V(H)}>\emph{\textbf{0}}$ and $\textbf{x}_w\le 0$ for any vertex $w$ adjacent to some vertex in $V(H)$, then $\lambda_2(G)\le \lambda_1(H)$.
The inequality is strict if $\textbf{x}_w< 0$ for some vertex $w$.
\end{lemma}

We can show that the second largest eigenvalue $\lambda_2(G)$ lies in the second largest disk, i.e., $\lambda_2(G)<\Delta_2$.

\begin{theorem}\label{thm:2th<d2}
    Let $G$ be a connected graph on $n$ vertices, then
    $$\lambda_2(G)<\Delta_2(G)-\frac{1}{n^2}.$$
    Roughly, $\lambda_2(G)<\Delta_2(G)$.
\end{theorem}
\begin{proof}
Let $\textit{\textbf{x}}$ be an eigenvector of the graph $G$ corresponding to $\lambda_2(G)$, and define
\begin{center}
    $N_+=\{v\in V(G):\textit{\textbf{x}}_v>0\}$,\ \ $N_-=\{v\in V(G):\textit{\textbf{x}}_v<0\}$\ \ and\ \ $N_0=V(G)\backslash(N_+\cup N_-)$.
\end{center}
Clearly, $N_+\not=\varnothing$ and $N_-\not=\varnothing$.
Let $G_1,\ldots, G_s$ be connected components of $G$ by removing all vertices in $N_0$ and edges in $E(N_+,N_-)$.
So $s\ge 2$.
From Lemma \ref{lem:connected by negative edges}, we get
$$\lambda_2(G)\le \min\big\{\rho(G_i):1\le i\le s\big\}.$$
Without loss of generality, let $G_1,\ldots,G_l$ be regular graphs, and $G_{l+1},\ldots,G_s$ be irregular graphs, $0\le l\le s$.
Write $n(G_i)$ as $n_i$, $1\le i\le s$.

From Lemmas \ref{lem:rho le max-degree} and \ref{lem:irregular gap}, we have, for $1\le i\le l$,
$$\rho(G_i)=\Delta_1(G_i)$$
and, for $l+1\le i\le s$,
$$\rho(G_i)<\Delta_1(G_i)-\frac{1}{n_i\big(D(G_i)+1\big)}.$$
For $1\le i\le l$, since $G$ is connected, we know $E\big(V(G_i),V(G)\backslash V(G_i)\big)\not=\varnothing$.
Let $u_i$ be a vertex of $G_i$ such that $N(u_i)\cap \big(V(G)\backslash V(G_i)\big)\not=\varnothing$ when $1\le i\le l$, and a vertex of $G_i$ such that $|N(u_i)\cap V(G_i)|=\Delta_1(G_i)$ when $l+1\le i\le s$.
Then $$\Delta_1(G_i)\le\begin{cases}
    & d(u_i)-1<d(u_i)-\frac{1}{n_i\big(D(G_i)+1\big)},\ \ 1\le i\le l;\\
    & d(u_i),\ \ l+1\le i\le s.
\end{cases}
$$
Thus, we can conclude
\begin{align*}
    \lambda_2(G)\le\min\big\{\rho(G_i):1\le i\le s\big\}&<\min\left\{d(u_i)-\frac{1}{n_i\big(D(G_i)+1\big)}:1\le i\le s\right\}\\
    &\le \min\left\{d(u_i)-\frac{1}{n^2_i}:1\le i\le s\right\}\\
    &\le \min\left\{d(u_i):1\le i\le s\right\}-\frac{1}{n^2}\\
    &\le \Delta_2(G)-\frac{1}{n^2},
\end{align*}
where the last inequality dues to $s\ge 2$.
This completes the proof.
\end{proof}

From Theorem \ref{thm:2th<d2}, the gap between $\Delta_2(G)$ and $\lambda_2(G)$ of an irregular connected graph $G$ is larger than $\frac{1}{n^2}$.
We believe the value $\frac{1}{n^2}$ is far from the optimal.
We remark our proof much relies on the bound in Lemma \ref{lem:irregular gap}. 
So a better gap of $\Delta_2(G)$ and $\lambda_2(G)$ may need a better bound for $\Delta_1(G)-\rho(G)$.

\section{Refinement on Hong's bound of the spectral radius}

The following result is a conjecture proposed by Guo, Wang and Li in \cite{GWL19} and has been proved recently by Sun and Das in \cite{Sun+Das}.

\begin{lemma}\label{SunDas}\cite{Sun+Das}
Let $G$ be a connected graph. For any vertex $v$ of $G$,
$$\rho(G)\le \sqrt{\rho^2(G-v)+2d(v)-1}.$$
The equality holds if and only if $G= S_n$ and $v$ is a pendent vertex, or $G= K_n$.
\end{lemma}

Given vertices $u$ and $v$, their distance $d_G(u, v)$ is the minimal length
of a path from $u$ to $v$. The eccentricity of a vertex $v$ in $G$ is the maximum distance from $v$ to all other vertices in $G$, denoted
by $e_G(v)$.

\begin{theorem}\label{inducedsubgraph}
Let $G$ be a connected graph with $n$ vertices and $m$ edges, and $G'$ be a proper induced subgraph of $G$ with $n'$ vertices and $m'$ edges. 
Then we have
$$\rho(G)\le \sqrt{2m-n-2m'+n'+\rho^2(G')},$$
and the equality holds if and only if  $G=S_n$ and $G'=S_{n'}$, or $G=K_n$ and $G'=K_{n'}$.
\end{theorem}
\begin{proof}
Let $V_1=V(G)-V(G')$. Then $|V_1|=n-n'$. Since $G$ is connected, there exists some spanning tree,
say $T$ in $G$. Suppose that $v$ is a vertex in $V(G')$. Let
$$U_i=\{u|d_T(u, v)=e_T(v)-i+1, 1\le i\le e_T(v) \}.$$

Let $S_i=V_1\cap U_i\ (1\le i\le e_T(v))$. It is easy to see that $\sum_{i=1}^{e_T(v)}|S_i|=|V_1|$.
Label the vertices of $\bigcup_{i=1}^{e_T(v)}S_i$ as $v_1, \ldots, v_{|S_1|}, v_{|S_1|+1}, \ldots, v_{|S_1|+|S_2|}, \ldots, v_{n-n'}$
successively. Then $G'$ can be considered as the graph obtained from $G$ by deleting vertices $v_1, v_2, \ldots, v_{n-n'}$ successively.
It is easy to see that for any $j \ (2\le i+1\le j\le n-n')$, $v_{j}$ is not an isolated vertex of
$G-v_1-v_2-\cdots -v_i$. Apply Lemma \ref{SunDas} to $G$ for $n-n'$ times successively, we have
\begin{align*}
\rho(G)&\le \sqrt{\rho^2(G-v_1)+2d_{G}(v_1)-1}\\
&\hspace{2mm}\vdots\\
&\le \sqrt{\rho^2(G')+2\sum_{i=1}^{n-n'}d_{G-v_1-\cdots-v_{i-1}}(v_i)-n+n'}.
\end{align*}

Note that $2m=2m'+2\sum_{i=1}^{n-n'}d_{G-v_1-\cdots-v_{i-1}}(v_i)$. Then we have $$\rho(G)\le \sqrt{2m-n-2m'+n'+\rho^2(G')}.$$
If the equality holds, then we have $\rho(G)= \sqrt{\rho^2(G-v_1)+2d_{G}(v_1)-1}$. From Lemma \ref{SunDas}, we have
$G= S_n$ and $v_k$ is a pendent vertex, or $G= K_n$. Thus we have $G=S_n$ and $G'=K_{1, n'-1}$ or $G=K_n$ and $G'=K_{n'}$.
If $G=S_n$ and $G'=K_{1, n'-1}$ or $G=K_n$ and $G'=K_{n'}$, it is easy to see that the equality holds.
\end{proof}

Setting $G'=P_2$ in Theorem \ref{inducedsubgraph}, we can directly deduce the following result by Hong in \cite{H88}, so Theorem \ref{inducedsubgraph} is a generalization of Hong's bound in terms of any induced subgraph.

\begin{lemma}\cite{H88}\label{lem:Hong's upper bound-1}
If $G$ is a connected graph with $n$ vertices and $m$ edges, then
$$\rho(G) \le \sqrt{2m - n + 1},$$
and the equality holds if and only if $G$ is a star or a complete graph.
\end{lemma}

Furthermore, using Theorem \ref{inducedsubgraph} and meticulous analysis on graph structures, we can refine Hong's bound in Lemma \ref{lem:Hong's upper bound-1}.
Let $S_n^e$ be the graph obtained by embedding an edge into a star $S_{n}$.

\begin{lemma}\label{Hongnotexist}\cite{Hongsystem}
If $G$ is a connected graph with $n\ge 5$ vertices and is neither the complete graph $K_n$
nor the graph $K^1_{n-1}$, then there exists some vertex, say $u\in V(G)$, such that $G-u$ is connected
and is  neither the complete graph $K_{n-1}$ nor the graph $K^1_{n-2}$.
\end{lemma}

\begin{lemma}\label{Guonotexist}
If $G$ is a connected graph with $n\ge 7$ vertices and $G\neq S_n, S^1_{n-1}, K_n$, $K^1_{n-1}$ and $S_n^e$,
then there exists some vertex, say $u\in V(G)$, such that $G-u$ is connected
and $G-u \neq S_{n-1}, S^1_{n-2}, K_{n-1}$, $K^1_{n-2}$ or $S_{n-1}^e$.
\end{lemma}
\begin{proof}
Let $G$ be a connected graph which satisfies the hypothesis of this lemma. We consider the following three cases, and look for the desired vertex $u$.

\vskip0.1in
\noindent\textbf{Case 1}. There exists a vertex $v$ such that $G-v = S_{n-1}$.
\vskip0.1in   
Suppose that $v$ is adjacent to the center of $S_{n-1}$. 
Then $v$ is adjacent to at least two pendant vertices of $S_{n-1}$; otherwise $G=S_n$ or $S_{n}^e$, a contradiction.
If $v$ is adjacent to exactly two pendant vertices of $S_{n-1}$, then $G$ has $n-4\ (\ge 1)$ pendant vertices.
Let $u$ be a pendant vertex of $G$. It is easy to see that $G-u$ is connected and
$G-u \neq S_{n-1}, S^1_{n-2}, K_{n-1}$, $K^1_{n-2}$ and $S_{n-1}^e$.
If $v$ is adjacent to at least three pendant vertices of $S_{n-1}$, then let $u$ be a vertex of $G$ with degree 2.
Then $G-u$ is connected and $G-u \neq S_{n-1}, S^1_{n-2}, K_{n-1}$, $K^1_{n-2}$ and $S_{n-1}^e$.

Suppose that $v$ is not adjacent to the center of $S_{n-1}$. 
Then $v$ is adjacent to at least two pendant vertices of $S_{n-1}$; otherwise $G=S^1_{n-1}$, a contradiction.
If $v$ is adjacent to exactly two pendant vertices of $S_{n-1}$, then there exists another pendant vertex $u$ of $G$
such that $G-u$ is connected and $G-u \neq S_{n-1}, S^1_{n-2}, K_{n-1}$, $K^1_{n-2}$ and $S_{n-1}^e$.
If $v$ is adjacent to at least three pendant vertices of $S_{n-1}$, then let $u$ be a vertex of $G$ with degree
2. It is easy to see that $G-u$ is connected and $G-u \neq S_{n-1}, S^1_{n-2}, K_{n-1}$, $K^1_{n-2}$ and $S_{n-1}^e$.

\vskip0.1in
\noindent\textbf{Case 2}. There exists a vertex $v$ such that  $G-v = S^1_{n-2}$.
\vskip0.1in 
Let $w_1$ be the vertex of $S^1_{n-2}$ with degree $n-3$,
 $w_2$ be the vertex of $S^1_{n-2}$ with degree 2, and $w_3$ be the pendant vertex which is adjacent to $w_2$.
 Let $w_4, w_5, \ldots, w_{n-1}$ be the pendant vertices of $S^1_{n-2}$ which are adjacent to $w_1$, respectively.
If $v$ is adjacent to $w_2$ or $w_3$ or some pendant vertex, say $w_4$, of $w_4, w_5, \ldots, w_{n-1}$,
then let $u=w_{n-1}$. It is easy to see that $G-u$ is connected and
$G-u \neq S_{n-1}, S^1_{n-2}, K_{n-1}$, $K^1_{n-2}$ and $S_{n-1}^e$. If
$v$ is only adjacent to $w_1$, then $G=S^1_{n-1}$, a contradiction.

\vskip0.1in
\noindent\textbf{Case 3}. There exists a vertex $v$ such that $G-v = S_{n-1}^e$.
\vskip0.1in 
Let $w_1$ be the vertex of $S_{n-1}^e$ with degree $n-2$ and $w_2$, $w_3$ be the two vertices of $S_{n-1}^e$ with degree 2, respectively. 
Let $u_1, u_2, \ldots, u_{n-4}$ be the pendant vertices of $S_{n-1}^e$, respectively. If $v$ is adjacent to $w_2$ or $w_3$, then let $u=u_1$.
It is easy to see that $G-u$ is connected and $G-u \neq S_{n-1},  S^1_{n-2}, K_{n-1}$, $K^1_{n-2}$ and $S_{n-1}^e$.
In the following, suppose that  $v$ is not adjacent to $w_2$ or $w_3$.
Then $v$ is adjacent to at least one vertex, say $u_1$, among $u_1, u_2, \ldots, u_{n-4}$; otherwise, $G = S_{n}^e$, a contradiction. Let $u=u_2$,
It is easy to see that $G-u$ is connected and $G-u \neq S_{n-1},  S^1_{n-2}, K_{n-1}$, $K^1_{n-2}$ and $S_{n-1}^e$.

From Cases 1, 2 and 3, we can suppose that for any vertex $v$ of $G$, $G-v \neq S_{n-1}, S^1_{n-2}, $ and $S_{n-1}^e$. From Lemma \ref{Hongnotexist},
there exists some vertex, say $u\in V(G)$, such that $G-u$ is connected and is  neither the complete graph $K_{n-1}$ nor
the graph $K^1_{n-2}$ since $G\neq K_n$ and $K^1_{n-1}$. The result follows.
\end{proof}

Let $G_{s, t}\ (6\le s\le 21, 1\le t\le 853)$
 be the graph corresponding to the graph with labeling ``$s-t$'' in the tables: Spectra of graphs with seven vertices (see \cite{Cvetkovicrecentresults}), where $s$ is the number of edges and $t$ infers $G_{s,t}$ is the $t$-th graph in the tables. 
For example, $G_{6, 10}=S^1_6$, $G_{6, 11}=S_7$, $G_{7, 44}=S_7^e$, $G_{16, 834}=K^1_{6}$, $G_{21, 853}=K_7$.
For simplification, we only present the graphs $G_{s,t}$ emerged in Lemma \ref{Sevenverticesgraph} and Theorem \ref{thm:refined upper bound-1} in Figure 1.

\vskip0.2in

\begin{figure}[h]
\centering
\begin{tikzpicture}[x=0.80mm, y=0.80mm, inner xsep=0pt, inner ysep=0pt, outer xsep=0pt, outer ysep=0pt]
\definecolor{L}{rgb}{0,0,0}
\definecolor{F}{rgb}{0,0,0}
\path[line width=0.20mm, draw=L, fill=F] (25.08,-250.07) circle (0.50mm);
\path[line width=0.20mm, draw=L] (25.08,-250.07) -- (34.87,-244.99);
\path[line width=0.20mm, draw=L] (34.87,-244.99) -- (44.94,-249.70);
\path[line width=0.20mm, draw=L, fill=F] (44.94,-249.70) circle (0.50mm);
\path[line width=0.20mm, draw=L, fill=F] (35.05,-244.99) circle (0.50mm);
\path[line width=0.20mm, draw=L] (35.05,-245.08) -- (35.01,-255.04);
\path[line width=0.20mm, draw=L] (35.05,-244.81) -- (30.06,-235.10);
\path[line width=0.20mm, draw=L, fill=F] (30.06,-235.10) circle (0.50mm);
\path[line width=0.20mm, draw=L] (30.16,-235.10) -- (40.21,-235.09);
\path[line width=0.20mm, draw=L, fill=F] (40.14,-234.92) circle (0.50mm);
\path[line width=0.20mm, draw=L] (40.14,-234.92) -- (34.87,-244.62);
\path[line width=0.20mm, draw=L] (40.14,-234.83) -- (45.03,-230.11);
\path[line width=0.20mm, draw=L, fill=F] (45.03,-230.11) circle (0.50mm);
\draw(-4.92,-263.75) node[anchor=base west]{\fontsize{12}{12}\selectfont $G_{6,9}$};
\draw(30.01,-263.84) node[anchor=base west]{\fontsize{12}{12}\selectfont $G_{7,43}$};
\draw(65.02,-263.84) node[anchor=base west]{\fontsize{12}{12}\selectfont $G_{8,111}$};
\path[line width=0.20mm, draw=L, fill=F] (60.09,-249.89) circle (0.50mm);
\path[line width=0.20mm, draw=L] (60.09,-249.89) -- (69.89,-244.81);
\path[line width=0.20mm, draw=L] (69.89,-244.81) -- (79.96,-249.52);
\path[line width=0.20mm, draw=L, fill=F] (79.96,-249.52) circle (0.50mm);
\path[line width=0.20mm, draw=L, fill=F] (70.07,-244.81) circle (0.50mm);
\path[line width=0.20mm, draw=L] (70.07,-244.62) -- (65.08,-234.92);
\path[line width=0.20mm, draw=L, fill=F] (65.08,-234.92) circle (0.50mm);
\path[line width=0.20mm, draw=L, fill=F] (75.16,-234.73) circle (0.50mm);
\path[line width=0.20mm, draw=L] (75.16,-234.73) -- (69.89,-244.44);
\path[line width=0.20mm, draw=L, fill=F] (70.01,-230.02) circle (0.50mm);
\path[line width=0.20mm, draw=L] (70.01,-230.02) -- (65.02,-234.92);
\path[line width=0.20mm, draw=L] (70.20,-230.02) -- (75.00,-234.92);
\path[line width=0.20mm, draw=L] (70.11,-229.74) -- (70.01,-254.93);
\draw(99.95,-263.75) node[anchor=base west]{\fontsize{12}{12}\selectfont $G_{9,218}$};
\draw(134.97,-263.84) node[anchor=base west]{\fontsize{12}{12}\selectfont $G_{10,350}$};
\draw(-5.11,-308.93) node[anchor=base west]{\fontsize{12}{12}\selectfont $G_{11,488}$};
\draw(29.91,-309.02) node[anchor=base west]{\fontsize{12}{12}\selectfont $G_{12,614}$};
\draw(64.84,-308.75) node[anchor=base west]{\fontsize{12}{12}\selectfont $G_{13,709}$};
\draw(100.04,-308.82) node[anchor=base west]{\fontsize{12}{12}\selectfont $G_{14,733}$};
\draw(134.98,-309.09) node[anchor=base west]{\fontsize{12}{12}\selectfont $G_{15,813}$};
\draw(-5.17,-353.83) node[anchor=base west]{\fontsize{12}{12}\selectfont $G_{16,833}$};
\draw(29.94,-353.57) node[anchor=base west]{\fontsize{12}{12}\selectfont $G_{17,844}$};
\draw(65.02,-353.99) node[anchor=base west]{\fontsize{12}{12}\selectfont $G_{18,849}$};
\draw(99.98,-353.94) node[anchor=base west]{\fontsize{12}{12}\selectfont $G_{19,851}$};
\draw(135.02,-353.92) node[anchor=base west]{\fontsize{12}{12}\selectfont $G_{20,852}$};
\path[line width=0.20mm, draw=L, fill=F] (0.01,-254.85) circle (0.50mm);
\path[line width=0.20mm, draw=L] (0.01,-254.85) -- (0.02,-235.01);
\path[line width=0.20mm, draw=L] (-0.03,-235.01) -- (-9.91,-230.01);
\path[line width=0.20mm, draw=L, fill=F] (-9.91,-230.01) circle (0.50mm);
\path[line width=0.20mm, draw=L] (-0.03,-235.01) -- (9.99,-230.08);
\path[line width=0.20mm, draw=L, fill=F] (9.99,-230.08) circle (0.50mm);
\path[line width=0.20mm, draw=L] (0.08,-244.86) -- (-9.97,-254.80);
\path[line width=0.20mm, draw=L, fill=F] (-9.90,-254.73) circle (0.50mm);
\path[line width=0.20mm, draw=L] (0.03,-245.00) -- (9.99,-254.91);
\path[line width=0.20mm, draw=L, fill=F] (9.99,-254.91) circle (0.50mm);
\path[line width=0.20mm, draw=L, fill=F] (0.04,-244.93) circle (0.50mm);
\path[line width=0.20mm, draw=L, fill=F] (0.02,-235.01) circle (0.50mm);
\path[line width=0.20mm, draw=L, fill=F] (35.01,-255.12) circle (0.50mm);
\path[line width=0.20mm, draw=L, fill=F] (69.97,-254.93) circle (0.50mm);
\path[line width=0.20mm, draw=L, fill=F] (95.08,-249.89) circle (0.50mm);
\path[line width=0.20mm, draw=L] (95.08,-249.89) -- (104.87,-244.81);
\path[line width=0.20mm, draw=L] (104.87,-244.81) -- (114.95,-249.52);
\path[line width=0.20mm, draw=L, fill=F] (114.95,-249.52) circle (0.50mm);
\path[line width=0.20mm, draw=L, fill=F] (105.06,-244.81) circle (0.50mm);
\path[line width=0.20mm, draw=L] (105.06,-244.62) -- (100.07,-234.92);
\path[line width=0.20mm, draw=L, fill=F] (100.07,-234.92) circle (0.50mm);
\path[line width=0.20mm, draw=L, fill=F] (110.14,-234.73) circle (0.50mm);
\path[line width=0.20mm, draw=L] (110.14,-234.73) -- (104.87,-244.44);
\path[line width=0.20mm, draw=L, fill=F] (105.00,-230.02) circle (0.50mm);
\path[line width=0.20mm, draw=L] (105.00,-230.02) -- (100.01,-234.92);
\path[line width=0.20mm, draw=L] (105.18,-230.02) -- (109.99,-234.92);
\path[line width=0.20mm, draw=L] (105.09,-229.74) -- (104.99,-254.93);
\path[line width=0.20mm, draw=L, fill=F] (104.96,-254.93) circle (0.50mm);
\path[line width=0.20mm, draw=L] (100.09,-234.93) -- (110.03,-234.98);
\path[line width=0.20mm, draw=L, fill=F] (130.05,-249.97) circle (0.50mm);
\path[line width=0.20mm, draw=L] (130.05,-249.97) -- (139.84,-244.88);
\path[line width=0.20mm, draw=L] (139.84,-244.88) -- (149.92,-249.60);
\path[line width=0.20mm, draw=L, fill=F] (149.92,-249.60) circle (0.50mm);
\path[line width=0.20mm, draw=L, fill=F] (140.03,-244.88) circle (0.50mm);
\path[line width=0.20mm, draw=L] (140.03,-244.70) -- (135.04,-235.00);
\path[line width=0.20mm, draw=L, fill=F] (135.04,-235.00) circle (0.50mm);
\path[line width=0.20mm, draw=L, fill=F] (145.11,-234.81) circle (0.50mm);
\path[line width=0.20mm, draw=L] (145.11,-234.81) -- (139.84,-244.51);
\path[line width=0.20mm, draw=L, fill=F] (139.97,-230.10) circle (0.50mm);
\path[line width=0.20mm, draw=L] (139.97,-230.10) -- (134.98,-235.00);
\path[line width=0.20mm, draw=L] (140.15,-230.10) -- (144.96,-235.00);
\path[line width=0.20mm, draw=L] (140.06,-229.82) -- (139.96,-255.00);
\path[line width=0.20mm, draw=L, fill=F] (139.93,-255.00) circle (0.50mm);
\path[line width=0.20mm, draw=L] (135.06,-235.01) -- (145.00,-235.06);
\path[line width=0.20mm, draw=L] (150.04,-249.85) -- (144.99,-234.95);
\path[line width=0.20mm, draw=L, fill=F] (9.97,-284.79) circle (0.50mm);
\path[line width=0.20mm, draw=L, fill=F] (-10.06,-284.83) circle (0.50mm);
\path[line width=0.20mm, draw=L] (-10.06,-284.83) -- (10.01,-284.81);
\path[line width=0.20mm, draw=L, fill=F] (2.51,-284.79) circle (0.50mm);
\path[line width=0.20mm, draw=L] (2.51,-284.79) -- (-3.87,-275.04);
\path[line width=0.20mm, draw=L] (-3.87,-275.04) -- (-9.99,-285.00);
\path[line width=0.20mm, draw=L, fill=F] (-3.83,-275.15) circle (0.50mm);
\path[line width=0.20mm, draw=L] (-3.76,-275.24) -- (-3.72,-294.84);
\path[line width=0.20mm, draw=L] (-3.76,-294.91) -- (-9.96,-284.95);
\path[line width=0.20mm, draw=L] (-3.69,-294.84) -- (2.62,-284.99);
\path[line width=0.20mm, draw=L, fill=F] (-3.72,-294.84) circle (0.50mm);
\path[line width=0.20mm, draw=L] (10.03,-284.95) -- (-3.76,-294.98);
\path[line width=0.20mm, draw=L] (10.18,-285.06) -- (-3.72,-275.14);
\path[line width=0.20mm, draw=L] (-10.07,-300.02) -- (-3.73,-295.00);
\path[line width=0.20mm, draw=L] (-3.73,-295.00) -- (3.13,-300.11);
\path[line width=0.20mm, draw=L, fill=F] (3.00,-300.04) circle (0.50mm);
\path[line width=0.20mm, draw=L, fill=F] (-10.00,-299.99) circle (0.50mm);
\path[line width=0.20mm, draw=L] (34.98,-279.82) -- (25.10,-289.64);
\path[line width=0.20mm, draw=L] (35.10,-279.82) -- (45.01,-289.77);
\path[line width=0.20mm, draw=L] (45.08,-289.80) -- (40.03,-299.75);
\path[line width=0.20mm, draw=L] (40.03,-299.78) -- (30.03,-299.78);
\path[line width=0.20mm, draw=L] (30.03,-299.78) -- (25.06,-289.72);
\path[line width=0.20mm, draw=L] (35.00,-279.78) -- (30.00,-299.78);
\path[line width=0.20mm, draw=L] (30.06,-299.75) -- (45.03,-289.83);
\path[line width=0.20mm, draw=L] (45.03,-289.83) -- (25.06,-289.75);
\path[line width=0.20mm, draw=L] (25.06,-289.75) -- (40.00,-299.67);
\path[line width=0.20mm, draw=L] (40.00,-299.67) -- (35.06,-279.83);
\path[line width=0.20mm, draw=L, fill=F] (35.06,-279.83) circle (0.50mm);
\path[line width=0.20mm, draw=L, fill=F] (24.97,-289.72) circle (0.50mm);
\path[line width=0.20mm, draw=L, fill=F] (45.00,-289.83) circle (0.50mm);
\path[line width=0.20mm, draw=L, fill=F] (40.06,-299.67) circle (0.50mm);
\path[line width=0.20mm, draw=L, fill=F] (30.03,-299.72) circle (0.50mm);
\path[line width=0.20mm, draw=L] (35.04,-279.91) -- (24.91,-274.90);
\path[line width=0.20mm, draw=L] (35.17,-280.05) -- (44.98,-275.08);
\path[line width=0.20mm, draw=L, fill=F] (44.98,-275.08) circle (0.50mm);
\path[line width=0.20mm, draw=L, fill=F] (25.05,-275.04) circle (0.50mm);
\path[line width=0.20mm, draw=L] (69.88,-279.99) -- (60.00,-289.81);
\path[line width=0.20mm, draw=L] (70.00,-279.99) -- (79.91,-289.94);
\path[line width=0.20mm, draw=L] (79.97,-289.97) -- (74.93,-299.92);
\path[line width=0.20mm, draw=L] (74.93,-299.95) -- (64.93,-299.95);
\path[line width=0.20mm, draw=L] (64.93,-299.95) -- (59.95,-289.89);
\path[line width=0.20mm, draw=L] (69.90,-279.95) -- (64.90,-299.95);
\path[line width=0.20mm, draw=L] (64.95,-299.92) -- (79.93,-290.00);
\path[line width=0.20mm, draw=L] (79.93,-290.00) -- (59.95,-289.92);
\path[line width=0.20mm, draw=L] (59.95,-289.92) -- (74.90,-299.84);
\path[line width=0.20mm, draw=L] (74.90,-299.84) -- (69.95,-280.00);
\path[line width=0.20mm, draw=L, fill=F] (69.95,-280.00) circle (0.50mm);
\path[line width=0.20mm, draw=L, fill=F] (59.87,-289.89) circle (0.50mm);
\path[line width=0.20mm, draw=L, fill=F] (79.90,-290.00) circle (0.50mm);
\path[line width=0.20mm, draw=L, fill=F] (74.95,-299.84) circle (0.50mm);
\path[line width=0.20mm, draw=L, fill=F] (64.93,-299.89) circle (0.50mm);
\path[line width=0.20mm, draw=L] (69.93,-280.08) -- (59.80,-275.07);
\path[line width=0.20mm, draw=L] (70.07,-280.22) -- (79.87,-275.25);
\path[line width=0.20mm, draw=L, fill=F] (79.87,-275.25) circle (0.50mm);
\path[line width=0.20mm, draw=L, fill=F] (59.94,-275.21) circle (0.50mm);
\path[line width=0.20mm, draw=L] (79.96,-275.04) -- (79.82,-290.02);
\path[line width=0.20mm, draw=L] (104.92,-283.09) -- (95.06,-289.73);
\path[line width=0.20mm, draw=L] (104.98,-283.06) -- (114.97,-289.86);
\path[line width=0.20mm, draw=L] (115.03,-289.89) -- (109.99,-299.84);
\path[line width=0.20mm, draw=L] (109.99,-299.87) -- (99.98,-299.87);
\path[line width=0.20mm, draw=L] (99.98,-299.87) -- (95.01,-289.81);
\path[line width=0.20mm, draw=L] (104.92,-283.06) -- (99.96,-299.87);
\path[line width=0.20mm, draw=L] (100.01,-299.84) -- (114.99,-289.92);
\path[line width=0.20mm, draw=L] (114.99,-289.92) -- (95.01,-289.84);
\path[line width=0.20mm, draw=L] (95.01,-289.84) -- (109.96,-299.75);
\path[line width=0.20mm, draw=L] (109.96,-299.75) -- (105.05,-283.16);
\path[line width=0.20mm, draw=L, fill=F] (94.93,-289.81) circle (0.50mm);
\path[line width=0.20mm, draw=L, fill=F] (114.96,-289.92) circle (0.50mm);
\path[line width=0.20mm, draw=L, fill=F] (110.01,-299.75) circle (0.50mm);
\path[line width=0.20mm, draw=L, fill=F] (99.98,-299.81) circle (0.50mm);
\path[line width=0.20mm, draw=L, fill=F] (105.06,-275.01) circle (0.50mm);
\path[line width=0.20mm, draw=L] (105.06,-275.01) -- (94.94,-289.60);
\path[line width=0.20mm, draw=L] (105.19,-275.16) -- (114.97,-290.06);
\path[line width=0.20mm, draw=L, fill=F] (104.92,-283.09) circle (0.50mm);
\path[line width=0.20mm, draw=L] (104.92,-283.09) -- (104.98,-275.14);
\path[line width=0.20mm, draw=L] (104.95,-283.16) -- (114.97,-279.87);
\path[line width=0.20mm, draw=L, fill=F] (114.97,-279.87) circle (0.50mm);
\path[line width=0.20mm, draw=L] (140.09,-299.88) -- (150.13,-290.12);
\path[line width=0.20mm, draw=L] (140.03,-299.92) -- (130.18,-289.85);
\path[line width=0.20mm, draw=L] (130.07,-289.96) -- (136.96,-283.01);
\path[line width=0.20mm, draw=L] (137.04,-283.01) -- (142.73,-283.05);
\path[line width=0.20mm, draw=L] (142.73,-282.94) -- (150.02,-290.18);
\path[line width=0.20mm, draw=L] (140.09,-299.92) -- (142.78,-282.89);
\path[line width=0.20mm, draw=L] (142.67,-283.05) -- (130.13,-290.01);
\path[line width=0.20mm, draw=L] (130.07,-289.85) -- (149.97,-290.07);
\path[line width=0.20mm, draw=L] (149.97,-290.07) -- (137.01,-283.01);
\path[line width=0.20mm, draw=L] (137.06,-283.06) -- (139.96,-299.82);
\path[line width=0.20mm, draw=L, fill=F] (140.09,-299.88) circle (0.50mm);
\path[line width=0.20mm, draw=L, fill=F] (140.00,-274.82) circle (0.50mm);
\path[line width=0.20mm, draw=L] (140.00,-274.82) -- (137.01,-282.96);
\path[line width=0.20mm, draw=L] (140.04,-274.82) -- (142.73,-283.05);
\path[line width=0.20mm, draw=L] (140.08,-274.93) -- (150.02,-290.01);
\path[line width=0.20mm, draw=L] (139.91,-274.87) -- (130.07,-289.96);
\path[line width=0.20mm, draw=L, fill=F] (130.35,-289.85) circle (0.50mm);
\path[line width=0.20mm, draw=L, fill=F] (150.02,-290.07) circle (0.50mm);
\path[line width=0.20mm, draw=L, fill=F] (142.78,-283.22) circle (0.50mm);
\path[line width=0.20mm, draw=L, fill=F] (149.88,-279.84) circle (0.50mm);
\path[line width=0.20mm, draw=L] (149.88,-279.84) -- (142.72,-283.06);
\path[line width=0.20mm, draw=L, fill=F] (137.01,-283.01) circle (0.50mm);
\path[line width=0.20mm, draw=L] (0.01,-344.85) -- (10.04,-335.09);
\path[line width=0.20mm, draw=L] (-0.06,-344.88) -- (-9.90,-334.82);
\path[line width=0.20mm, draw=L] (-10.01,-334.93) -- (-3.13,-327.97);
\path[line width=0.20mm, draw=L] (-3.05,-327.97) -- (2.64,-328.02);
\path[line width=0.20mm, draw=L] (2.64,-327.91) -- (9.93,-335.15);
\path[line width=0.20mm, draw=L] (0.01,-344.88) -- (2.70,-327.85);
\path[line width=0.20mm, draw=L] (2.58,-328.02) -- (-9.96,-334.98);
\path[line width=0.20mm, draw=L] (-10.01,-334.82) -- (9.88,-335.04);
\path[line width=0.20mm, draw=L] (9.88,-335.04) -- (-3.08,-327.97);
\path[line width=0.20mm, draw=L] (-3.02,-328.02) -- (-0.13,-344.79);
\path[line width=0.20mm, draw=L, fill=F] (0.01,-344.85) circle (0.50mm);
\path[line width=0.20mm, draw=L, fill=F] (-0.09,-319.78) circle (0.50mm);
\path[line width=0.20mm, draw=L] (-0.09,-319.78) -- (-3.08,-327.92);
\path[line width=0.20mm, draw=L] (-0.04,-319.78) -- (2.64,-328.02);
\path[line width=0.20mm, draw=L] (-0.01,-319.90) -- (9.93,-334.98);
\path[line width=0.20mm, draw=L] (-0.18,-319.84) -- (-10.01,-334.93);
\path[line width=0.20mm, draw=L, fill=F] (-9.74,-334.82) circle (0.50mm);
\path[line width=0.20mm, draw=L, fill=F] (9.93,-335.04) circle (0.50mm);
\path[line width=0.20mm, draw=L, fill=F] (2.70,-328.18) circle (0.50mm);
\path[line width=0.20mm, draw=L, fill=F] (9.79,-324.80) circle (0.50mm);
\path[line width=0.20mm, draw=L] (9.79,-324.80) -- (2.63,-328.02);
\path[line width=0.20mm, draw=L, fill=F] (-3.08,-327.97) circle (0.50mm);
\path[line width=0.20mm, draw=L] (9.81,-324.81) -- (9.93,-335.11);
\path[line width=0.20mm, draw=L] (105.60,-344.69) -- (115.00,-340.00);
\path[line width=0.20mm, draw=L] (105.57,-344.71) -- (95.04,-339.98);
\path[line width=0.20mm, draw=L] (95.04,-339.98) -- (101.97,-333.93);
\path[line width=0.20mm, draw=L] (109.04,-333.99) -- (114.98,-339.98);
\path[line width=0.20mm, draw=L] (105.59,-344.72) -- (109.03,-333.99);
\path[line width=0.20mm, draw=L] (109.03,-334.00) -- (95.06,-339.90);
\path[line width=0.20mm, draw=L] (95.03,-339.99) -- (114.98,-340.00);
\path[line width=0.20mm, draw=L] (114.95,-339.98) -- (102.01,-333.96);
\path[line width=0.20mm, draw=L] (102.02,-333.89) -- (105.55,-344.74);
\path[line width=0.20mm, draw=L] (105.23,-319.91) -- (102.00,-333.93);
\path[line width=0.20mm, draw=L] (105.36,-319.87) -- (109.04,-333.98);
\path[line width=0.20mm, draw=L] (105.36,-319.86) -- (115.34,-339.91);
\path[line width=0.20mm, draw=L] (105.20,-319.91) -- (95.04,-339.88);
\path[line width=0.20mm, draw=L] (105.43,-325.97) -- (105.57,-344.66);
\path[line width=0.20mm, draw=L, fill=F] (105.63,-344.69) circle (0.50mm);
\path[line width=0.20mm, draw=L] (102.00,-333.99) -- (109.03,-333.96);
\path[line width=0.20mm, draw=L] (105.48,-326.06) -- (102.05,-333.98);
\path[line width=0.20mm, draw=L] (105.48,-326.03) -- (108.96,-333.93);
\path[line width=0.20mm, draw=L] (105.49,-326.02) -- (114.87,-339.88);
\path[line width=0.20mm, draw=L] (105.46,-326.03) -- (95.06,-339.98);
\path[line width=0.20mm, draw=L, fill=F] (109.06,-334.06) circle (0.50mm);
\path[line width=0.20mm, draw=L, fill=F] (102.04,-333.95) circle (0.50mm);
\path[line width=0.20mm, draw=L, fill=F] (95.05,-340.01) circle (0.50mm);
\path[line width=0.20mm, draw=L, fill=F] (114.94,-339.97) circle (0.50mm);
\path[line width=0.20mm, draw=L, fill=F] (105.48,-326.05) circle (0.50mm);
\path[line width=0.20mm, draw=L, fill=F] (105.29,-319.92) circle (0.50mm);
\path[line width=0.20mm, draw=L] (140.68,-344.59) -- (150.09,-339.91);
\path[line width=0.20mm, draw=L] (140.65,-344.62) -- (130.12,-339.88);
\path[line width=0.20mm, draw=L] (130.12,-339.88) -- (137.05,-333.83);
\path[line width=0.20mm, draw=L] (144.12,-333.89) -- (150.07,-339.88);
\path[line width=0.20mm, draw=L] (140.67,-344.63) -- (144.11,-333.89);
\path[line width=0.20mm, draw=L] (144.11,-333.90) -- (130.14,-339.81);
\path[line width=0.20mm, draw=L] (130.12,-339.89) -- (150.07,-339.91);
\path[line width=0.20mm, draw=L] (150.03,-339.88) -- (137.09,-333.86);
\path[line width=0.20mm, draw=L] (137.10,-333.80) -- (140.63,-344.64);
\path[line width=0.20mm, draw=L] (140.32,-319.82) -- (137.09,-333.83);
\path[line width=0.20mm, draw=L] (140.44,-319.77) -- (144.13,-333.88);
\path[line width=0.20mm, draw=L] (140.44,-319.76) -- (150.43,-339.81);
\path[line width=0.20mm, draw=L] (140.29,-319.82) -- (130.12,-339.79);
\path[line width=0.20mm, draw=L] (140.49,-320.07) -- (140.65,-344.56);
\path[line width=0.20mm, draw=L, fill=F] (140.72,-344.59) circle (0.50mm);
\path[line width=0.20mm, draw=L] (137.08,-333.90) -- (144.11,-333.86);
\path[line width=0.20mm, draw=L] (140.56,-325.96) -- (137.13,-333.88);
\path[line width=0.20mm, draw=L] (140.56,-325.93) -- (144.05,-333.83);
\path[line width=0.20mm, draw=L] (140.57,-325.92) -- (149.95,-339.78);
\path[line width=0.20mm, draw=L] (140.54,-325.93) -- (130.14,-339.88);
\path[line width=0.20mm, draw=L, fill=F] (144.15,-333.96) circle (0.50mm);
\path[line width=0.20mm, draw=L, fill=F] (137.12,-333.85) circle (0.50mm);
\path[line width=0.20mm, draw=L, fill=F] (130.14,-339.91) circle (0.50mm);
\path[line width=0.20mm, draw=L, fill=F] (150.02,-339.87) circle (0.50mm);
\path[line width=0.20mm, draw=L, fill=F] (140.56,-325.95) circle (0.50mm);
\path[line width=0.20mm, draw=L, fill=F] (140.38,-319.82) circle (0.50mm);
\path[line width=0.20mm, draw=L] (34.92,-326.11) -- (25.52,-330.80);
\path[line width=0.20mm, draw=L] (34.96,-326.09) -- (45.48,-330.82);
\path[line width=0.20mm, draw=L] (45.49,-330.82) -- (38.56,-336.87);
\path[line width=0.20mm, draw=L] (31.49,-336.81) -- (25.54,-330.82);
\path[line width=0.20mm, draw=L] (34.94,-326.08) -- (31.49,-336.81);
\path[line width=0.20mm, draw=L] (31.50,-336.80) -- (45.47,-330.90);
\path[line width=0.20mm, draw=L] (45.49,-330.81) -- (25.54,-330.80);
\path[line width=0.20mm, draw=L] (25.58,-330.82) -- (38.51,-336.84);
\path[line width=0.20mm, draw=L] (38.51,-336.91) -- (34.97,-326.06);
\path[line width=0.20mm, draw=L, fill=F] (34.89,-326.11) circle (0.50mm);
\path[line width=0.20mm, draw=L] (38.53,-336.80) -- (31.50,-336.84);
\path[line width=0.20mm, draw=L] (35.05,-344.74) -- (38.48,-336.82);
\path[line width=0.20mm, draw=L] (35.05,-344.77) -- (31.56,-336.87);
\path[line width=0.20mm, draw=L] (35.03,-344.78) -- (25.66,-330.92);
\path[line width=0.20mm, draw=L] (35.07,-344.77) -- (45.47,-330.82);
\path[line width=0.20mm, draw=L, fill=F] (31.46,-336.74) circle (0.50mm);
\path[line width=0.20mm, draw=L, fill=F] (38.48,-336.85) circle (0.50mm);
\path[line width=0.20mm, draw=L, fill=F] (45.47,-330.79) circle (0.50mm);
\path[line width=0.20mm, draw=L, fill=F] (25.58,-330.83) circle (0.50mm);
\path[line width=0.20mm, draw=L, fill=F] (35.05,-344.75) circle (0.50mm);
\path[line width=0.20mm, draw=L, fill=F] (34.92,-319.68) circle (0.50mm);
\path[line width=0.20mm, draw=L] (34.92,-319.68) -- (25.60,-330.72);
\path[line width=0.20mm, draw=L] (34.99,-319.61) -- (45.31,-331.00);
\path[line width=0.20mm, draw=L] (34.90,-344.83) -- (35.03,-326.51);
\path[line width=0.20mm, draw=L] (69.89,-326.21) -- (60.49,-330.90);
\path[line width=0.20mm, draw=L] (69.92,-326.19) -- (80.45,-330.92);
\path[line width=0.20mm, draw=L] (80.45,-330.92) -- (73.52,-336.97);
\path[line width=0.20mm, draw=L] (66.45,-336.91) -- (60.51,-330.92);
\path[line width=0.20mm, draw=L] (69.90,-326.18) -- (66.46,-336.91);
\path[line width=0.20mm, draw=L] (66.46,-336.91) -- (80.43,-331.00);
\path[line width=0.20mm, draw=L] (80.46,-330.91) -- (60.51,-330.90);
\path[line width=0.20mm, draw=L] (60.54,-330.92) -- (73.48,-336.94);
\path[line width=0.20mm, draw=L] (73.47,-337.01) -- (69.94,-326.16);
\path[line width=0.20mm, draw=L, fill=F] (69.86,-326.21) circle (0.50mm);
\path[line width=0.20mm, draw=L] (73.49,-336.91) -- (66.46,-336.94);
\path[line width=0.20mm, draw=L] (70.01,-344.84) -- (73.44,-336.92);
\path[line width=0.20mm, draw=L] (70.01,-344.87) -- (66.53,-336.97);
\path[line width=0.20mm, draw=L] (70.00,-344.88) -- (60.62,-331.02);
\path[line width=0.20mm, draw=L] (70.03,-344.87) -- (80.43,-330.92);
\path[line width=0.20mm, draw=L, fill=F] (66.43,-336.84) circle (0.50mm);
\path[line width=0.20mm, draw=L, fill=F] (73.45,-336.95) circle (0.50mm);
\path[line width=0.20mm, draw=L, fill=F] (80.44,-330.89) circle (0.50mm);
\path[line width=0.20mm, draw=L, fill=F] (60.55,-330.93) circle (0.50mm);
\path[line width=0.20mm, draw=L, fill=F] (70.01,-344.85) circle (0.50mm);
\path[line width=0.20mm, draw=L, fill=F] (69.88,-319.78) circle (0.50mm);
\path[line width=0.20mm, draw=L] (69.88,-319.78) -- (60.57,-330.82);
\path[line width=0.20mm, draw=L] (69.96,-319.71) -- (80.27,-331.11);
\path[line width=0.20mm, draw=L] (69.87,-344.93) -- (69.93,-319.91);
\end{tikzpicture}%
\caption{Graphs $G_{s,t}$ in Lemma \ref{Sevenverticesgraph} and Theorem \ref{thm:refined upper bound-1}}\label{F1}
\end{figure}
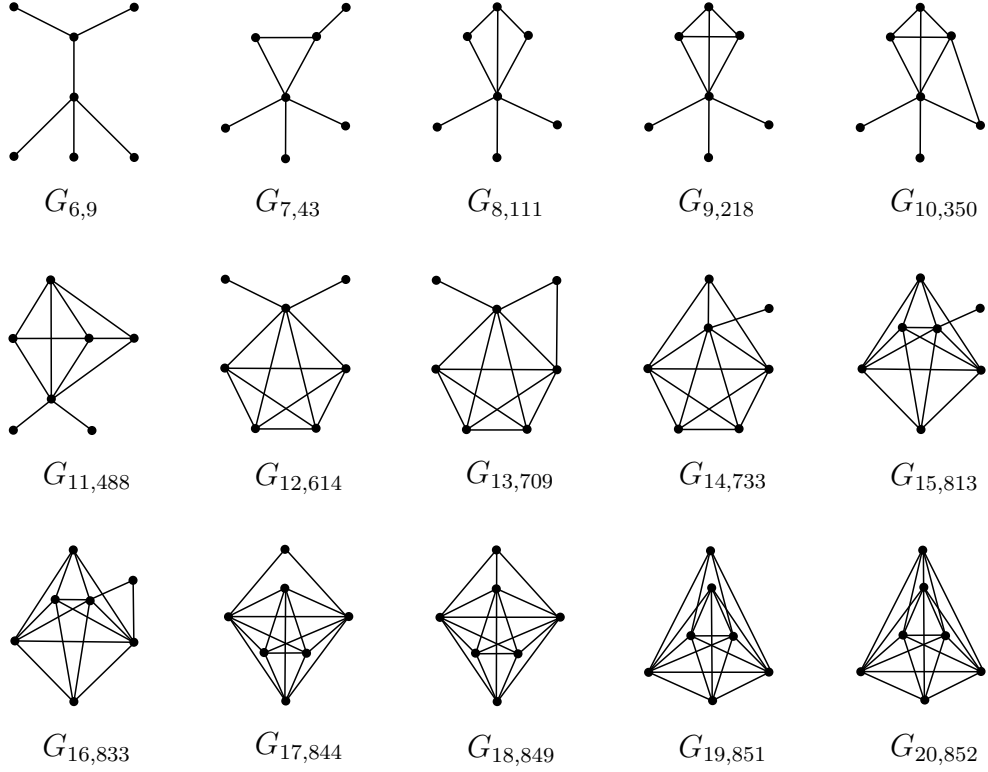

\vskip0.2in

\begin{lemma}\label{Sevenverticesgraph}
Let $G$ be a connected graph with $7$ vertices and $m$ edges and $G\neq S_7, S^1_6, K_7$, $K^1_{6}$ and $S_7^e$.
Then we have the following:

$(1)$ If $m=6$, then $\rho(G)\le \rho(G_{6, 9})\approx 2.17533$, with equality if and only if $G=G_{6, 9}$;

$(2)$ If $m=7$, then $\rho(G)\le \rho(G_{7, 43})\approx 2.59440$, with equality if and only if $G=G_{7, 43}$;

$(3)$ If $m=8$, then $\rho(G)\le \rho(G_{8, 111})\approx 2.94388$, with equality if and only if $G=G_{8, 111}$;

$(4)$ If $m=9$, then $\rho(G)\le \rho(G_{9, 218})\approx 3.27307$, with equality if and only if $G=G_{9, 218}$;

$(5)$ If $m=10$, then $\rho(G)\le \rho(G_{10, 350})\approx 3.48767$, with equality if and only if $G=G_{10, 350}$;

$(6)$ If $m=11$, then $\rho(G)\le \rho(G_{11, 488})\approx 3.77846$, with equality if and only if $G=G_{11, 488}$;

$(7)$ If $m=12$, then $\rho(G)\le \rho(G_{12, 614})\approx 4.10548$, with equality if and only if $G=G_{12, 614}$;

$(8)$ If $m=13$, then $\rho(G)\le \rho(G_{13, 709})\approx 4.25327$, with equality if and only if $G=G_{13, 709}$;

$(9)$ If $m=14$, then $\rho(G)\le \rho(G_{14, 773})\approx 4.47415$, with equality if and only if $G=G_{14, 773}$;

$(10)$ If $m=15$, then $\rho(G)\le \rho(G_{15, 813})\approx 4.74146$, with equality if and only if $G=G_{15, 813}$;

$(11)$ If $m=16$, then $\rho(G)\le \rho(G_{16, 833})\approx 4.85952$, with equality if and only if $G=G_{16, 833}$;

$(12)$ If $m=17$, then $\rho(G)\le \rho(G_{17, 844})\approx 5.13555$, with equality if and only if $G=G_{17, 844}$;

$(13)$ If $m=18$, then $\rho(G)\le \rho(G_{18, 849})\approx 5.29654$, with equality if and only if $G=G_{18, 849}$;

$(14)$ If $m=19$, then $\rho(G)\le \rho(G_{19, 851})\approx 5.50331$, with equality if and only if $G=G_{19, 851}$;

$(15)$ If $m=20$, then $\rho(G)\le \rho(G_{20, 852})\approx 5.74166$, with equality if and only if $G=G_{20, 852}$.
\end{lemma}
\begin{proof}
Since $G\not=K_7$, we know $6\le m\le 20$.
The result follows by checking the tables: Spectra of graphs with seven vertices (see \cite{Cvetkovicrecentresults}).
\end{proof}

Recall that $\mathcal{H}_n=\{S_{n}, S^1_{n-1}, K_n, K^1_{n-1}\}$.
Now we are going to give the refinement of Lemma \ref{lem:Hong's upper bound-1}.

\begin{theorem}\label{thm:refined upper bound-1}
If $G$ is a connected graph with $n \ge 10$ vertices and $m$ edges, and
$G \notin \mathcal{H}_n$, then
$$\rho(G)<\sqrt{2m-n}.$$
\end{theorem}
\begin{proof}
Let $G$ be a connected graph which satisfies the hypothesis of this theorem. 
If $G=S_n^e$, then let
$G'=S_9^e$. It is easy to see that $|V(S_9^e)|=|E(S_9^e)|=9$ and $\rho(S_9^e)=3$.
From Theorem \ref{inducedsubgraph},  we have
$$\rho(G)< \sqrt{2m-n-18+9+\rho^2(S_9^e)}=\sqrt{2m-n}.$$

In the following, we suppose that $G\neq S_n^e$. 
Apply Lemma \ref{Guonotexist} to $G$ for $n-7$ times,
then there exists a connected graph, say $G'$, with 7 vertices such that
$G'\neq S_7, S^1_{6}, K_7$, $K^1_{6}$ and $S_7^e$. Then $6\le |E(G')|\le 20$.
From Theorem \ref{inducedsubgraph},
we have
$$\rho(G)<\sqrt{2m-n-2m(G')+7+\rho^2(G')}.$$
Combining with Lemma \ref{Sevenverticesgraph}, we get

\noindent(1) If $m(G')=6$, then $\rho(G)< \sqrt{2m-n-5+\rho^2(G_{6, 9})}\le \sqrt{2m-n-5+4.7321}$;

\noindent(2) If $m(G')=7$, then $\rho(G)< \sqrt{2m-n-7+\rho^2(G_{7, 43})}\le \sqrt{2m-n-7+6.7310}$;

\noindent(3) If $m(G')=8$, then $\rho(G)< \sqrt{2m-n-9+\rho^2(G_{8, 111})}\le \sqrt{2m-n-9+8.6665}$;

\noindent(4) If $m(G')=9$, then $\rho(G)< \sqrt{2m-n-11+\rho^2(G_{9, 218})}\le \sqrt{2m-n-11+10.7130}$;

\noindent(5) If $m(G')=10$, then $\rho(G)< \sqrt{2m-n-13+\rho^2(G_{10, 350})}\le \sqrt{2m-n-13+12.1640}$;

\noindent(6) If $m(G')=11$, then $\rho(G)< \sqrt{2m-n-15+\rho^2(G_{11, 488})}\le \sqrt{2m-n-15+14.2768}$;

\noindent(7) If $m(G')=12$, then $\rho(G)< \sqrt{2m-n-17+\rho^2(G_{12, 614})}\le \sqrt{2m-n-17+16.8550}$;

\noindent(8) If $m(G')=13$, then $\rho(G)< \sqrt{2m-n-19+\rho^2(G_{13, 709})}\le \sqrt{2m-n-19+18.0904}$;

\noindent(9) If $m(G')=14$, then $\rho(G)< \sqrt{2m-n-21+\rho^2(G_{14, 773})}\le \sqrt{2m-n-21+20.0181}$;

\noindent(10) If $m(G')=15$, then $\rho(G)< \sqrt{2m-n-23+\rho^2(G_{15, 813})}\le \sqrt{2m-n-23+22.4815}$;

\noindent(11) If $m(G')=16$, then $\rho(G)< \sqrt{2m-n-25+\rho^2(G_{16, 833})}\le \sqrt{2m-n-25+23.6150}$;

\noindent(12) If $m(G')=17$, then $\rho(G)< \sqrt{2m-n-27+\rho^2(G_{17, 844})}\le \sqrt{2m-n-27+26.3739}$;

\noindent(13) If $m(G')=18$, then $\rho(G)< \sqrt{2m-n-29+\rho^2(G_{18, 849})}\le \sqrt{2m-n-29+28.0534}$;

\noindent(14) If $m(G')=19$, then $\rho(G)< \sqrt{2m-n-31+\rho^2(G_{19, 851})}\le \sqrt{2m-n-31+30.2865}$;

\noindent(15) If $m(G')=20$, then $\rho(G)< \sqrt{2m-n-33+\rho^2(G_{20, 852})}\le \sqrt{2m-n-33+32.9667}$.
Thus, we obtain $\rho(G)<\sqrt{2m-n}$, completing the proof.
\end{proof}

\begin{corollary}
If $G$ is a connected simple graph with $n\ge 10$ vertices and $m$
edges, then $\sqrt{2m-n}<\rho(G)\le\sqrt{2m-n+1}$ if and only if $G\in \mathcal{H}_n$.
\end{corollary}
\begin{proof}
From Theorem \ref{thm:refined upper bound-1}, it suffices to check $\rho(G)>\sqrt{2m-n}$ for the graph $G$ in $\mathcal{H}_n$.
From Lemma \ref{lem:Hong's upper bound-1}, it has $\rho(G)=\sqrt{2m-n+1}$ if and only if $G$ is $S_n$ or $K_n$.
Note that $S^1_{n-1}$ contains $S_{n-1}$ as an induced subgraph, and $K^1_{n-1}$
contains $K_{n-1}$ as an induced subgraph. 
Then from the well-known Perron-Frobenius Theorem, we have
$$\rho(S^1_{1, n-1})>\rho(S_{n-1})=\sqrt{n-2}=\sqrt{2m(S^1_{n-1})-n},$$
and
$$\rho(K^1_{n-1})>\rho(K_{n-1})=n-2=\sqrt{2m(K^1_{n-1})-n}.$$
So the result holds.
\end{proof}

With the help of Sagemath, we list all exceptional graphs of order $n\le9$ with spectral radius at least $\sqrt{2m-n}$.

\begin{proposition}\label{prop:rho-except graph}
    Let $n\le 9$ and $G$ be a graph with $n$ vertices and $m$ edges.
    If $\rho(G)\ge \sqrt{2m-n}$, then $G\in \mathcal{H}_n$ or $G$ is one of graphs $R_1$--$R_{15}$ shown in Figure 2.
\end{proposition}

\vskip0.2in

\setlength{\unitlength}{0.5pt}
\begin{center}
\begin{picture}(760,540)
\put(0,480){\circle*{6}}
\put(40,520){\circle*{6}}
\qbezier(0,480)(20,500)(40,520)
\put(80,480){\circle*{6}}
\qbezier(40,520)(60,500)(80,480)
\put(40,440){\circle*{6}}
\qbezier(0,480)(20,460)(40,440)
\qbezier(40,440)(60,460)(80,480)
\qbezier(40,520)(40,480)(40,440)
\put(320,480){\circle*{6}}
\put(400,480){\circle*{6}}
\qbezier(320,480)(360,480)(400,480)
\put(320,440){\circle*{6}}
\qbezier(320,440)(320,460)(320,480)
\put(400,440){\circle*{6}}
\qbezier(400,480)(400,460)(400,440)
\put(360,520){\circle*{6}}
\qbezier(360,520)(340,500)(320,480)
\qbezier(360,520)(380,500)(400,480)
\put(240,520){\circle*{6}}
\put(160,520){\circle*{6}}
\qbezier(240,520)(200,520)(160,520)
\put(200,480){\circle*{6}}
\qbezier(240,520)(220,500)(200,480)
\qbezier(160,520)(180,500)(200,480)
\put(160,440){\circle*{6}}
\qbezier(200,480)(180,460)(160,440)
\put(240,440){\circle*{6}}
\qbezier(200,480)(220,460)(240,440)
\put(480,480){\circle*{6}}
\put(520,520){\circle*{6}}
\qbezier(480,480)(500,500)(520,520)
\put(520,440){\circle*{6}}
\qbezier(480,480)(500,460)(520,440)
\qbezier(520,440)(520,480)(520,520)
\put(560,480){\circle*{6}}
\qbezier(520,440)(540,460)(560,480)
\qbezier(520,520)(540,500)(560,480)
\put(560,540){\circle*{6}}
\qbezier(520,520)(540,530)(560,540)
\put(680,520){\circle*{6}}
\put(680,440){\circle*{6}}
\qbezier(680,520)(680,480)(680,440)
\put(720,480){\circle*{6}}
\qbezier(680,520)(700,500)(720,480)
\qbezier(720,480)(700,460)(680,440)
\put(760,480){\circle*{6}}
\qbezier(680,520)(720,500)(760,480)
\qbezier(760,480)(720,460)(680,440)
\put(640,480){\circle*{6}}
\qbezier(680,520)(660,500)(640,480)
\qbezier(640,480)(660,460)(680,440)
\put(40,360){\circle*{6}}
\put(0,320){\circle*{6}}
\qbezier(40,360)(20,340)(0,320)
\put(40,320){\circle*{6}}
\qbezier(0,320)(20,320)(40,320)
\put(80,320){\circle*{6}}
\qbezier(40,320)(60,320)(80,320)
\put(40,280){\circle*{6}}
\qbezier(80,320)(60,300)(40,280)
\put(200,360){\circle*{6}}
\put(160,320){\circle*{6}}
\qbezier(200,360)(180,340)(160,320)
\put(240,320){\circle*{6}}
\qbezier(200,360)(220,340)(240,320)
\put(200,280){\circle*{6}}
\qbezier(160,320)(180,300)(200,280)
\qbezier(240,320)(220,300)(200,280)
\qbezier(200,360)(200,320)(200,280)
\qbezier(160,320)(200,320)(240,320)
\put(280,320){\circle*{6}}
\qbezier(200,360)(240,340)(280,320)
\qbezier(200,280)(240,300)(280,320)
\put(360,360){\circle*{6}}
\put(320,320){\circle*{6}}
\qbezier(360,360)(340,340)(320,320)
\put(400,320){\circle*{6}}
\qbezier(360,360)(380,340)(400,320)
\put(360,280){\circle*{6}}
\qbezier(320,320)(340,300)(360,280)
\qbezier(400,320)(380,300)(360,280)
\qbezier(360,360)(360,320)(360,280)
\qbezier(320,320)(360,320)(400,320)
\put(440,320){\circle*{6}}
\qbezier(360,360)(400,340)(440,320)
\qbezier(400,320)(420,320)(440,320)
\qbezier(440,320)(400,300)(360,280)
\put(520,360){\circle*{6}}
\put(480,320){\circle*{6}}
\qbezier(520,360)(500,340)(480,320)
\put(520,320){\circle*{6}}
\qbezier(520,360)(520,340)(520,320)
\qbezier(480,320)(500,320)(520,320)
\put(560,340){\circle*{6}}
\qbezier(520,320)(540,330)(560,340)
\put(552,288){\circle*{6}}
\qbezier(520,320)(536,304)(552,288)
\put(500,280){\circle*{6}}
\qbezier(520,320)(510,300)(500,280)
\put(680,360){\circle*{6}}
\put(640,320){\circle*{6}}
\qbezier(680,360)(660,340)(640,320)
\put(720,320){\circle*{6}}
\qbezier(680,360)(700,340)(720,320)
\put(680,280){\circle*{6}}
\qbezier(640,320)(660,300)(680,280)
\qbezier(720,320)(700,300)(680,280)
\qbezier(680,360)(680,320)(680,280)
\qbezier(640,320)(680,320)(720,320)
\put(751,351){\circle*{6}}
\qbezier(720,320)(735,336)(751,351)
\put(751,289){\circle*{6}}
\qbezier(720,320)(735,305)(751,289)
\put(40,160){\circle*{6}}
\put(0,120){\circle*{6}}
\qbezier(40,160)(20,140)(0,120)
\put(80,120){\circle*{6}}
\qbezier(40,160)(60,140)(80,120)
\qbezier(0,120)(40,120)(80,120)
\put(40,80){\circle*{6}}
\qbezier(40,160)(40,120)(40,80)
\qbezier(0,120)(20,100)(40,80)
\qbezier(80,120)(60,100)(40,80)
\put(80,180){\circle*{6}}
\qbezier(40,160)(60,170)(80,180)
\put(80,60){\circle*{6}}
\qbezier(40,80)(60,70)(80,60)
\put(200,160){\circle*{6}}
\put(240,160){\circle*{6}}
\qbezier(200,160)(220,160)(240,160)
\put(280,120){\circle*{6}}
\qbezier(240,160)(260,140)(280,120)
\put(160,120){\circle*{6}}
\qbezier(200,160)(180,140)(160,120)
\put(200,80){\circle*{6}}
\qbezier(160,120)(180,100)(200,80)
\put(240,80){\circle*{4}}
\qbezier(200,80)(220,80)(240,80)
\qbezier(240,80)(260,100)(280,120)
\qbezier(200,160)(220,120)(240,80)
\qbezier(200,160)(200,120)(200,80)
\qbezier(200,160)(240,140)(280,120)
\qbezier(240,160)(200,140)(160,120)
\qbezier(240,160)(220,120)(200,80)
\qbezier(240,160)(240,120)(240,80)
\qbezier(280,120)(240,100)(200,80)
\qbezier(240,80)(200,100)(160,120)
\put(410,160){\circle*{6}}
\put(338,160){\circle*{6}}
\qbezier(410,160)(374,160)(338,160)
\put(374,124){\circle*{6}}
\qbezier(410,160)(392,142)(374,124)
\qbezier(338,160)(356,142)(374,124)
\put(415,100){\circle*{6}}
\qbezier(374,124)(394,112)(415,100)
\put(392,72){\circle*{6}}
\qbezier(374,124)(383,98)(392,72)
\put(355,72){\circle*{6}}
\qbezier(374,124)(364,98)(355,72)
\put(332,100){\circle*{6}}
\qbezier(374,124)(353,112)(332,100)
\put(554,160){\circle*{6}}
\put(486,160){\circle*{6}}
\qbezier(554,160)(520,160)(486,160)
\put(520,126){\circle*{6}}
\qbezier(554,160)(537,143)(520,126)
\qbezier(486,160)(503,143)(520,126)
\put(567,109){\circle*{6}}
\qbezier(520,126)(543,118)(567,109)
\put(550,85){\circle*{6}}
\qbezier(520,126)(535,106)(550,85)
\put(520,74){\circle*{6}}
\qbezier(520,126)(520,100)(520,74)
\put(491,85){\circle*{6}}
\qbezier(520,126)(505,106)(491,85)
\put(473,109){\circle*{6}}
\qbezier(520,126)(496,118)(473,109)
\put(646,160){\circle*{6}}
\put(714,160){\circle*{6}}
\qbezier(646,160)(680,160)(714,160)
\put(680,126){\circle*{6}}
\qbezier(714,160)(697,143)(680,126)
\qbezier(646,160)(663,143)(680,126)
\put(720,94){\circle*{6}}
\qbezier(680,126)(700,110)(720,94)
\put(699,73){\circle*{6}}
\qbezier(680,126)(689,100)(699,73)
\put(662,73){\circle*{6}}
\qbezier(680,126)(671,100)(662,73)
\put(640,94){\circle*{6}}
\qbezier(680,126)(660,110)(640,94)
\put(634,126){\circle*{6}}
\qbezier(680,126)(657,126)(634,126)
\put(726,126){\circle*{6}}
\qbezier(680,126)(703,126)(726,126)
\put(27,406){$R_1$}
\put(187,406){$R_2$}
\put(347,406){$R_3$}
\put(508,406){$R_4$}
\put(668,406){$R_5$}
\put(27,246){$R_6$}
\put(187,246){$R_7$}
\put(347,246){$R_8$}
\put(508,246){$R_9$}
\put(667,246){$R_{10}$}
\put(26,31){$R_{11}$}
\put(201,31){$R_{12}$}
\put(357,31){$R_{13}$}
\put(508,31){$R_{14}$}
\put(665,31){$R_{15}$}
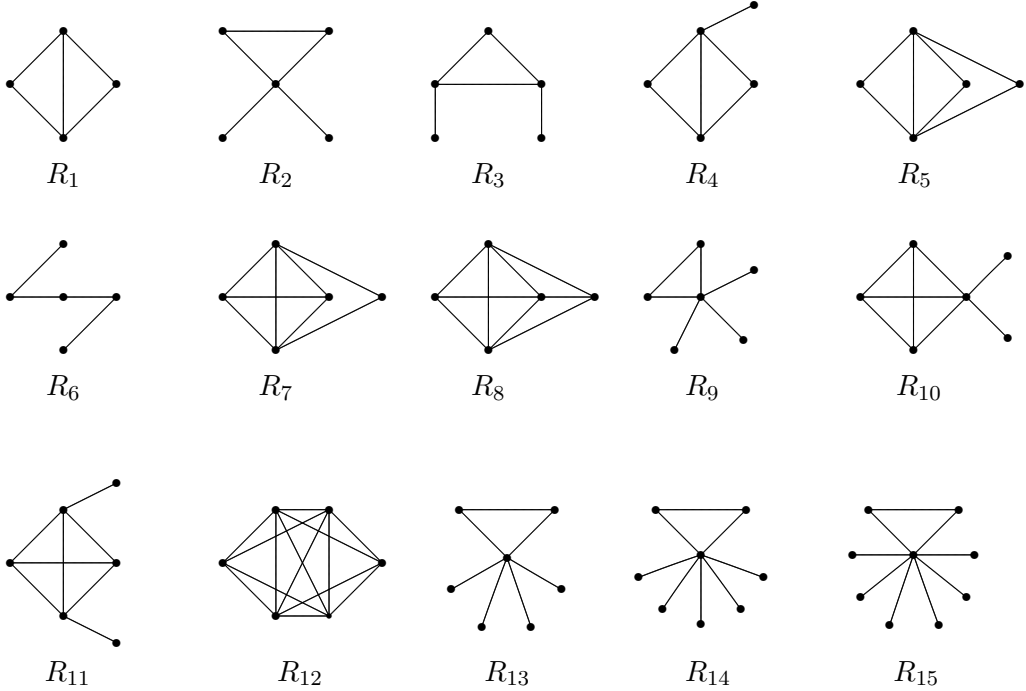
\put(50,-20){Figure 2: Small graphs with spectral radius at least $\sqrt{2m-n}$}
\end{picture}
\end{center}

\vskip0.5in

\section{Sharp upper bound of the second largest eigenvalue}\label{section:4}

We need to introduce the next critical tool, which gives the effect on the second largest eigenvalue by suitably deleting edge.

\begin{lemma}\cite{WG26+}\label{lem:remove edge between N+ and N-}
Let $\textbf{x}$ be an eigenvector of a graph $G$ corresponding to $\lambda_2(G)$, and $uv$ be an edge of $G$. If $\textbf{x}_u\textbf{x}_v<0$, then $\lambda_2(G-uv)> \lambda_2(G)$.
\end{lemma}
If $G_1$ and $G_2$ are two disjoint graphs, $u$ is a vertex of $G_1$, and $v$ is a vertex of $G_2$, we denote by $G_1uvG_2$ the graph from $G_1\cup G_2$ by connecting the edge $uv$.
Let $\mathcal{J}(G_1,G_2)$ be a family of graphs obtained from $G_1$ and $G_2$ by adding a new vertex $u$ and connecting $u$ to a vertex of each $G_i$ for $i=1,2$.

The following lemma is a computational result, so we put the proof into Appendix \ref{appendix A}.
\begin{lemma}\label{lem:four special graphs-2}
Let $n_1\ge 10$, $n_2\ge 10$ be two integers, $H_1\in \mathcal{H}_{n_1}$ and $H_2\in \mathcal{H}_{n_2}$. 
If $G$ is the graph $H_1uvH_2$ with $n=n_1+n_2$ vertices and $m$ edges, such that $\textbf{x}\big|_{V(H_1)}>0$ and $\textbf{x}\big|_{V(H_2)}<0$, where $\textbf{x}$ is an eigenvector of a graph $G$ corresponding to $\lambda_2(G)$, then
    $$\lambda_2(G)<\sqrt{m-\frac{n}{2}-\frac{1}{2}},$$
except that $G$ is $S_{\frac{n}{2}}uvS_{\frac{n}{2}}$ such that $u$ and $v$ are pendent vertices of $S_{\frac{n}{2}}$.
\end{lemma}



\begin{theorem}\label{thm:m.n-1}
    Let $G$ be a connected graph of order $n$ and size $m$, and $G$ is not isomorphic to $S_{\frac{n}{2}}uvS_{\frac{n}{2}}$, then
    $$\lambda_2(G)\le \sqrt{m-\frac{n}{2}-\frac{1}{2}},$$
    and the equality holds if and only if $G\in \mathcal{J}(K_{\frac{n-1}{2}},K_{\frac{n-1}{2}})\cup\mathcal{J}(S_{\frac{n-1}{2}},S_{\frac{n-1}{2}})\cup\mathcal{J}(K_{r+1},S_{r^2+1})$, where $r$ is an integer and $r^2+r+3=n$.
\end{theorem}
\begin{proof}
Let $G$ be a graph with the maximum second largest eigenvalue among connected graphs of order $n$ and size at most $m$.
Let $\textit{\textbf{x}}$ be an eigenvector of $G$ corresponding to $\lambda_2(G)$, and define
\begin{center}
    $N_+=\{v\in V(G):\textit{\textbf{x}}_v>0\}$,\ \ $N_-=\{v\in V(G):\textit{\textbf{x}}_v<0\}$\ \ and\ \ $N_0=V(G)\backslash(N_+\cup N_-)$.
\end{center}

We assert that each edge in $E(N_+,N_-)$ is a cut edge.
Suppose on the contrary that $uv\in E(N_+,N_-)$ is not a cut edge of $G$.
We know that $G-uv$ is a connected graph of order $n$ and size $m(G-uv)=m(G)-1<m$.
Then from Lemma \ref{lem:remove edge between N+ and N-} we have $\lambda_2(G-uv)>\lambda_2(G)$, contradicting the maximality of $\lambda_2(G)$.

Duo to Perron-Frobenius Theorem, an eigenvector with elements of the same sign is respecting to the spectral radius of $G$.
So we have $N_+\not=\varnothing$ and $N_-\not=\varnothing$.
We now distinguish two cases depending on $E(N_+,N_-)$.
\vskip0.1in
\noindent\textbf{Case 1}. $E(N_+,N_-)\not=\varnothing$.
\vskip0.1in
Let $H_1, H_2,\ldots, H_s$ be connected components of $G-E(N_+,N_-)$, then
$s=e(N_+,N_-)+1$.
From Lemma \ref{lem:connected by negative edges}, we have \footnote{If there exist vertices of $N_0$ in some $H_i$, then these vertices have neighbors in both $V(H_i)\cap N_-$ and $V(H_i)\cap N_+$.
So $V(H_i)\cap N_-\not=\varnothing$ and $V(H_i)\cap N_0\not=\varnothing$, Then $G[V(H_i)\cap N_+]$ is a proper subgraph of $H_i$. From Lemma \ref{lem:connected by negative edges}, we get $\lambda_2(G)<\rho(G[V(H_i)\cap N_+])<\rho(H_i)$.}
\begin{align}\label{eq:noempty-upper bound}
    \lambda_2(G)<\min\left\{\rho(H_i):\ i=1,2,\ldots,s\right\}.
\end{align}
For $i=1,2,\ldots,s$, from Lemma \ref{lem:Hong's upper bound-1}, we have
\begin{align}\label{eq:Hong's upper bound for i}
    \rho(H_i)\le\sqrt{2m_i-n_i+1},
\end{align}
where $m_i=m(H_i)$ and $n_i=n(H_i)$.

If there is a connected component $H_i$ such that $\rho(H_i)\le \sqrt{m-\frac{n}{2}-\frac{1}{2}}$, then $\lambda_2(G)<\sqrt{m-\frac{n}{2}-\frac{1}{2}}$ by Eq.\,(\ref{eq:noempty-upper bound}).
Next suppose that $\rho(H_i)> \sqrt{m-\frac{n}{2}-\frac{1}{2}}$ for all $i=1,2,\ldots,s$.
Then by Eq.\,(\ref{eq:Hong's upper bound for i}) we obtain
\begin{align*}
    \sum_{i=1}^s(2m_i-n_i+1)>sm-\frac{sn}{2}-\frac{s}{2}.
\end{align*}
This follows that
\begin{align*}
    \left(\frac{s}{2}-1\right)n=\frac{s}{2}n-\sum_{i=1}^s{n_i}
    &>sm-2\sum_{i=1}^s{m_i}-\frac{3}{2}s\\
    &\ge sm(G)-2\big(m(G)-e(N_+,N_-)\big)-\frac{3}{2}s\\
    &=sm(G)-2\big(m(G)-s+1\big)-\frac{3}{2}s\\
    &=(s-2)m(G)+\frac{1}{2}s-2\\
    &\ge (s-2)(n-1)+\frac{1}{2}s-2\\
    &=(s-2)n-\frac{1}{2}s,
\end{align*}
which implies $s<2+\frac{2}{n-1}$. So $s=2$ by mentioning that $s\ge 2$ and $n\ge 10$.
Then $e(N_+,N_-)=1$, and the graph $G$ can be viewed as $H_1$ and $H_2$ by embedding an edge between a vertex of $H_1$ and a vertex of $H_2$.

If $\rho(H_1)<\sqrt{2m_1-n_1}$ or $\rho(H_2)<\sqrt{2m_2-n_2}$, then, without loss of generality, let $\rho(H_1)<\sqrt{2m_1-n_1}$, and $\rho(H_2)\le\sqrt{2m_2-n_2+1}$ from Lemma \ref{lem:Hong's upper bound-1}.
So
$$2m_1-n_1+2m_2-n_2+1>\rho^2(H_1)+\rho^2(H_2)>2\left(m-\frac{n}{2}-\frac{1}{2}\right)\ge 2m(G)-n-1.$$
Since $n=n_1+n_2$ and $m(G)=m_1+m_2+1$, we compute that $2m_1-n_1+2m_2-n_2+1=2(m(G)-1)-n+1=2m(G)-n-1$, contradicting the previous inequality.

Thus, we have $\rho(H_1)\ge\sqrt{2m_1-n_1}$ and $\rho(H_2)\ge\sqrt{2m_2-n_2}$.
If $n_1\ge 10$ and $n_2\ge 10$, then $H_1\in\mathcal{H}_{n_1}$ and $H_2\in\mathcal{H}_{n_2}$ from Theorem \ref{thm:refined upper bound-1}.
Recall that $e(N_+,N_-)=1$. We know $N_0=\varnothing$. So $\textit{\textbf{x}}\big|_{V(H_1)}>0$ and $\textit{\textbf{x}}\big|_{V(H_2)}<0$.
From Lemma \ref{lem:four special graphs-2}, we get
$$\lambda_2(G)<\sqrt{m-\frac{n}{2}-\frac{1}{2}},$$
except the graph $S_{\frac{n}{2}}uvS_{\frac{n}{2}}$ such that $u$ and $v$ are pendent vertices of $S_{\frac{n}{2}}$.

If one of $n_1$ and $n_2$ is at most $9$, we let $n_1=\min\{n_1,n_2\}\le 9$ without loss of generality.
From Proposition \ref{prop:rho-except graph} and Theorem \ref{thm:refined upper bound-1}, we have $H_1\in \mathcal{H}_{n_1}\cup\{R_i:i=1,\ldots,15\}$, and $H_2\in \mathcal{H}_{n_1}\cup\{R_i:i=1,\ldots,15\}$ when $n_2\le 9$, or $H_2\in \mathcal{H}_{n_1}$ when $n_2\ge 10$.
Moreover, we can show $n_2$ is small when $n_2\ge 10$.
Since
\begin{align*}
    \rho^2(H_1)>m-\frac{n}{2}-\frac{1}{2}&\ge m(G)-\frac{n}{2}-\frac{1}{2}\\
               &=m_1+m_2+1-\frac{n_1+n_2}{2}-\frac{1}{2}\\
               &=m_1-\frac{n_1-1}{2}+\left(m_2-\frac{n_2}{2}\right)\\
               &\ge m_1-\frac{n_1+1}{2}+\frac{n_2}{2},
\end{align*}
where the last inequality dues to $m_2\ge n_2-1$,
we get $n_2<2\rho^2(H_1)-(2m_1-n_1-1)$, that is,
$$n_2\le2\rho^2(H_1)-(2m_1-n_1).$$
It is checked by AutoGraphix-III that all graphs $H_1uvH_2$ in this situation \footnote{For example, if $H_1=K_4$, then $n_2\le 2\rho^2(K_4)-\big(2m(K_4)-n(K_4)\big)=10$. Then $H_2\in \mathcal{H}_4\cup\mathcal{H}_5\cup\mathcal{H}_6\cup\mathcal{H}_7\cup\mathcal{H}_8\cup\mathcal{H}_9\cup\mathcal{H}_{10}\cup \{R_i:i=1,\ldots,15\}$. Actually, if we not reduce $m_2$ to $n_2-1$, we can get $2m_2-n_2\le 2\rho^2(H_1)-(2m_1-n_1)$. So many possible graphs for $H_2$ should not be discussed.
For the situation $H_1=K_4$, we not need to consider $R_8,R_{12}, K^1_5, K_6, K^1_6,K_{7},K^1_7,K_8,K^1_8,K_9,K^1_9,K_{10}$.}, except $S_{\frac{n}{2}}uvS_{\frac{n}{2}}$, have $\rho(H_1uvH_2)<\sqrt{m(G)-\frac{n}{2}-\frac{1}{2}}$.

\vskip0.1in
\noindent\textbf{Case 2}. $E(N_+,N_-)=\varnothing$.
\vskip0.1in
We know that $N_0\not=\varnothing$.
Let $G[N_0]=H_1\cup H_2\cup\cdots\cup H_p$ and $G-N_0=G_1\cup G_2\cup \cdots \cup G_q$, where $H_j$ ($1\le j\le p$) and $G_i$ ($1\le i\le q$) are connected components. 
Then
\begin{align}\label{eq:case2-=}
    \lambda_2(G)=\rho(G_1)=\cdots=\rho(G_q)
\end{align}
from Lemma \ref{lem:null cut-set}.
For $i=1,2,\ldots,q$, write $m(G_i)$ and $n(G_i)$ as $m_i$ and $n_i$, respectively.
From Lemma \ref{lem:Hong's upper bound-1}, we have
\begin{align}\label{eq:case2-hong}
    \rho(G_i)\le\sqrt{2m_i-n_i+1},
\end{align}
and the equality holds if and only if $G_i$ is a star or a complete graph.

When $\lambda_2(G)\le0$, we know $G$ is a complete multipartite graph.
Then $\lambda_2(G)<\sqrt{m-\frac{n}{2}-\frac{1}{2}}$ since $m\ge m(G)\ge n-1$.
We may let $\lambda_2(G)>0$.

\vskip0.2in
\begin{claim}\label{clm:sum}
    It has $$\sum_{i=1}^q(2m_i-n_i+1)\le q\left(m-\frac{n}{2}-\frac{1}{2}\right),$$ and the equality holds only if $p=1$, $q=2$ and $n(H_1)=1$.
\end{claim}
\noindent\textbf{Proof of Claim \ref{clm:sum}.}
To show this claim, it suffices to obtain only the equality holds, and $p=1$, $q=2$ and $n(H_1)=1$, provided $\sum_{i=1}^q(2m_i-n_i+1)\ge q\left(m-\frac{n}{2}-\frac{1}{2}\right)$.
Then
$$2\sum_{i=1}^q{m_i}-\sum_{i=1}^q{n_i}+q\ge q\left(m-\frac{n}{2}-\frac{1}{2}\right)\ge q\left(m(G)-\frac{n}{2}-\frac{1}{2}\right).$$
Since
$$m(G)=\sum_{i=1}^q{m_i}+\sum_{j=1}^p{m(H_i)}+\sum_{j=1}^p{e(H_i,G-N_0)}\ \ \textrm{and}\ \ n=\sum_{i=1}^q{n_i}+\sum_{j=1}^p{n(H_j)},$$
we have
\begin{align*}
    2m(G)-2\sum_{j=1}^p{m(H_i)}-2\sum_{j=1}^p{e(H_i,G-N_0)}-n+\sum_{j=1}^p{n(H_j)}+q\ge qm(G)-\frac{q}{2}n-\frac{q}{2}.
\end{align*}
It is clear that $e(H_i,G-N_0)\ge 2$ since $e(H_i,N_+)\ge 1$ and $e(H_i,N_-)\ge 1$, and $m(H_i)\ge n(H_i)-1$ since $H_i$ is connected.
So we can conclude that
\begin{align}
    \left(\frac{q}{2}-1\right)n
    &\ge (q-2)m(G)+2\sum_{j=1}^p{e(H_i,G-N_0)}-\frac{3}{2}q+\sum_{j=1}^p{\big(2m(H_i)-n(H_i)\big)}\nonumber\\
    &\ge (q-2)m(G)+4p-\frac{3}{2}q+\sum_{j=1}^p\big(n(H_i)-2\big)\nonumber\\
    &\ge (q-2)m(G)+3p-\frac{3}{2}q\label{eq:mid-inequ1}
\end{align}

Since $\lambda_2(G)>0$, from the eigen-equation $A(G)\textit{\textbf{x}}=\lambda_2(G)\textit{\textbf{x}}$, we have
$n_i\ge 2$, and so $n\ge 2q+1$.
Let $m(G)=n-1+c$, where $c$ is the cycle space dimension.
By Eq.\,(\ref{eq:mid-inequ1}), we get
\begin{align*}
    \left(\frac{q}{2}-1\right)n
    &\ge (q-2)(n-1+c)+3p-\frac{3}{2}q,
\end{align*}
which follows
\begin{align}\label{eq:mid-inequ2}
    0&\ge \left(\frac{q}{2}-1\right)n-(c-1)(q-2)+3p-\frac{3}{2}q\nonumber\\
     &\ge\left(\frac{q}{2}-1\right)(2q+1)-(c-1)(q-2)+3p-\frac{3}{2}q\nonumber\\
     &=q(q-4)+c(q-2)+3p+1.
\end{align}
Mention that $c\ge 0$ and $p\ge 1$. 
If $q\ge 4$, then $2c+3p+1\le 0$ by Eq.\,(\ref{eq:mid-inequ2}), a contradiction.
If $q=3$, then $c+3p-2\le 0$ by Eq.\,(\ref{eq:mid-inequ2}), a contradiction.
So $q=2$. Then $p\le 1$ by Eq.\,(\ref{eq:mid-inequ2}).
This means $p=1$ and all inequlities above turn into equalities.
Thus, $\sum_{i=1}^q(2m_i-n_i+1)= q\left(m-\frac{n}{2}-\frac{1}{2}\right)$, and $n(H_1)=1$ from Eq.\,(\ref{eq:mid-inequ1}). $\hfill\blacksquare$ 
\vskip0.2in

If $\rho(G_i)<\sqrt{m-\frac{n}{2}-\frac{1}{2}}$ for some $i\in\{1,\ldots,q\}$, then we have $\lambda_2(G)=\rho(G_i)<\sqrt{m-\frac{n}{2}-\frac{1}{2}}$ by Eq.\,(\ref{eq:case2-=}).
If $\rho(G_i)\ge\sqrt{m-\frac{n}{2}-\frac{1}{2}}$ for all $i\in\{1,\ldots,q\}$, then by Eqs.\,(\ref{eq:case2-=}) and (\ref{eq:case2-hong}) we have
\begin{align}\label{eq:case2-sum-hong}
    q\lambda_2^2(G)=\sum_{i=1}^q{\rho^2(G_i)}\le \sum_{i=1}^q\big(2m_i-n_i+1\big).
\end{align}
This follows that
$$q\left(m-\frac{n}{2}-\frac{1}{2}\right)\le\sum_{i=1}^q\big(2m_i-n_i+1\big).$$
Combining with Claim \ref{clm:sum}, we have $q\left(m-\frac{n}{2}-\frac{1}{2}\right)=\sum_{i=1}^q\big(2m_i-n_i+1\big)$, and then $p=1$, $q=2$ and $n(H_1)=1$.

So $G\in\mathcal{J}(G_1,G_2)$ with $\rho(G_1)=\rho(G_2)$.
Moreover, the equality in Eq.\,(\ref{eq:case2-sum-hong}) gives that $\rho(G_i)=\sqrt{2m_i-n_i+1}$, $i=1,2$, which indicates that $G_i$ is a star or a complete graph from Lemma \ref{lem:Hong's upper bound-1}.
Since $\rho(K_x)=x-1$ and $\rho(S_{x})=\sqrt{x-1}$, we have $G_1$ and $G_2$ satisfy one of conditions: (i) $G_1=G_2=K_{\frac{n-1}{2}}$, (ii) $G_1=G_2=S_{\frac{n-1}{2}}$, or (iii) $G_1=K_{r+1}$ and $G_2=S_{r^2+1}$ with $r^2+r+3=n$.

Furthermore, if $H\in \mathcal{J}(K_{\frac{n-1}{2}},K_{\frac{n-1}{2}})\cup\mathcal{J}(S_{\frac{n-1}{2}},S_{\frac{n-1}{2}})\cup\mathcal{J}(K_{r+1},S_{r^2+1})$ with $r^2+r+3=n$, then, by Cauchy Interlacing theorem, we get
\begin{align*}
    \lambda_2(H)=\begin{cases}
        &\ \ \ \frac{n-1}{2}-1,\ \ \mathrm{if}\ H\in \mathcal{J}(K_{\frac{n-1}{2}},K_{\frac{n-1}{2}});\\
        &\sqrt{\frac{n-1}{2}-1},\ \ \mathrm{if}\ H\in \mathcal{J}(S_{\frac{n-1}{2}},S_{\frac{n-1}{2}});\\
        &\hspace{15mm} r\hspace{0.2mm},\ \ \mathrm{if}\ H\in\mathcal{J}(K_{r+1},S_{r^2+1}).
    \end{cases}
\end{align*}
So we can check $\lambda_2(H)=\sqrt{m(H)-\frac{n(H)}{2}-\frac{1}{2}}$.
This completes the proof.
\end{proof}

\section{Concluding remarks}

It is well known that $\lambda_1(G)\le \Delta_1(G)$ (see Gerschgorin’s Theorem or \cite{BH12}).
Currently, the best estimate for $\Delta_1(G)$ and $\lambda_1(G)$ of irregular graphs is at least $\frac{1}{n(D+1)}$, obtained by Cioab\u{a}, Gregory and Nikiforov \cite{CGN07}.
In this paper, we determine $\lambda_2(G)$ is less than $\Delta_2(G)$, and the gap of them is at least $\frac{1}{n^2}$.
We are interested in the optimal difference of $\Delta_2(G)$ and $\lambda_2(G)$.

Theorem \ref{thm:m.n-1} gives that $\lambda_2(G)\le \sqrt{m-\frac{n}{2}-\frac{1}{2}}$ except $G$ is a special tree. 
Or roughly, it shows $$\lambda_2(G)\le\sqrt{m-\frac{n}{2}}$$ for any connected graph $G$ of order $n$ and size $m$.
And it is proved
$$\lambda_1(G)\le \sqrt{2m-n+1}$$
from Hong's theorem \cite{H88}.
So this motivates us to guess
\begin{align}\label{eq:conjecture}
    \lambda_k(G)\le\sqrt{\frac{2m-n}{k}+c_k},
\end{align}
for a connected graph $G$ of order $n$ and size $m$ with the $k$-th largest eigenvalue $\lambda_k(G)$, where $c_k$ is a constant about $k$.

For a tree $T$, it is proved by Shao \cite{S91} and Chen \cite{C07} that 
$$\lambda_k(T)\le \sqrt{\frac{n}{k}-c_k}$$
with $c_k\in (-2,-1]$.
Thus, the bound in Eq.\,(\ref{eq:conjecture}) is validate for trees.

Regarding the $k$-th largest eigenvalue of a graph, among outerplanar graphs $G$ of order $n$, a recent nice work by Brooks, Gu, Hyatt, Linz and Lu \cite{BGHLL25} shows
$$\lambda_k(G)\le\sqrt{\frac{n}{k}}+1+O\left(\frac{1}{\sqrt{n}}\right).$$
This also confirms the bound in Eq.\,(\ref{eq:conjecture}) for outerplanar graphs when $m\ge \big(1+o(1)\big)n$.

\bibliographystyle{plain}

\begin{thebibliography}{99}
\vskip 0.2in

\bibitem{A86}
N. Alon, Eigenvalues and Expanders, \textit{Combinatorica} \textbf{6} (1986) 83--96.

\bibitem{AM85}
N. Alon, V. D. Milman, $\lambda_1$, isoperimetric inequalities for graphs and superconcentrators, 
\textit{J. Combin. Theory Ser. B} \textbf{38} (1985) 73--88.

\bibitem{AS00}
N. Alon, B. Sudakov, Bipartite subgraphs and the smallest eigenvalue, \textit{Combin. Probab. Comput.} \textbf{9}(1) (2000) 1--12.

\bibitem{AH10}
M. Aouchiche, P. Hansen, A survey of automated conjectures in spectral graph theory, \textit{Linear Algebra Appl.} \textbf{432} (2010) 2293--2322.

\bibitem{B25}
I. Balla, Equiangular lines via matrix projection, \textit{Adv. Math.} \textbf{482} (2025) 110620.

\bibitem{B65}
H. E. Bell, Gerschgorin's Theorem and the Zeros of Polynomials, \textit{Amer. Math. Monthly} \textbf{72} (1965) 292--295.

\bibitem{BH12}
A. E. Brouwer, W. H. Haemers, \textit{Spectra of Graphs}, Springer, 2012. 

\bibitem{BH85}
R. A. Brualdi, A. J. Hoffman, On the spectral radius of $(0, 1)$ matrices, \textit{Linear Algebra Appl.} \textbf{65} (1985) 133--146.

\bibitem{BD84}
R. C. Brigham, R. D. Dutton, Bounds on Graph Spectra, \textit{J. Combin. Theory Ser. B} \textbf{37} (1984) 228--234.

\bibitem{BGHLL25}
G. Brooks, M. Gu, J. Hyatt, W. Linz, L. Lu, On the maximum second eigenvalue of outerplanar graphs, \textit{Discrete Math.} \textbf{348} (2025) 114416.

\bibitem{C07}
J. Chen,  The proof on the conjecture of extremal graphs for the $k$th eigenvalues of trees, \textit{Linear Algebra Appl.} \textbf{426} (2007) 12--21.

\bibitem{CGN07}
S. M. Cioab\u{a}, D. A. Gregory, V. Nikiforov, Extreme eigenvalues of nonregular graphs, \textit{J. Combin. Theory Ser. B} \textbf{97} (2007) 483--486.

\bibitem{H82}
D. M. Cvetkovi\'c, On graphs whose second largest eigenvalue does not exceed 1, \textit{Publ. Inst. Math. (Belgr.)} \textbf{31} (1982) 15--20.

\bibitem{Cvetkovicrecentresults} D. M. Cvetkovi\'{c}, M. Doob, I. Gutman, A. Torga\u{s}ev,
Recent Results in the Theory of Graph Spectra, \textit{Ann. Discrete Math.} 36, 1988.

\bibitem{F04}
J. Friedman, A proof of Alon's second eigenvalue conjecture, Memoirs of AMS. pp. 720--724 in: \textit{STOC ’03: Proc. 35th annual ACM symp. on Theory of computing}, San Diego, 2003, ACM, New York, 2003.

\bibitem{FKS89}
J. Friedman, J. Kahn, E. Szemeredi, On the second eigenvalue in random regular graphs, \textit{Proc. 21st ACM STOC} (1989) 587--598. 

\bibitem{GWL19}
J. Guo, Z. Wang, X. Li, Sharp upper bounds of the spectral radius of a graph, \textit{Discrete Math.} \textbf{342} (2019) 2559--2563.

\bibitem{H88}
Y. Hong, A bound on the spectral radius of graphs, \textit{Linear Algebra Appl.} \textbf{108} (1988) 135--139.

\bibitem{Hongsystem}
Y. Hong, On the least eigenvalue of a graph, \textit{J. Syst. Sci. Math. Sci.} \textbf{6}(3) (1993) 269--272.

\bibitem{HSF01}
Y. Hong, J.L. Shu, K.F. Fang, A sharp upper bound of the spectral radius of graphs, \textit{J. Comb. Theory, Ser. B} \textbf{81} (2001) 177--183.

\bibitem{H19}
H. Huang, Induced subgraphs of hypercubes and a proof of the sensitivity conjecture, \textit{Ann. of Math.} \textbf{190}(2) (2019) 949--955.

\bibitem{JTYZZ21}
Z. Jiang, J. Tidor, Y. Yao, S. Zhang, Y. Zhao, Equiangular lines with a fixed angle, \textit{Ann. Math.} \textbf{194} (2021) 729--743.

\bibitem{LCS24}
M. Liu, C. Chen, Z. Stani\'c, Connected ($K_4-e$)-free graphs whose second largest eigenvalue does not exceed 1, \textit{Eur. J. Comb.} \textbf{115} (2024) 103775.

\bibitem{LPS88}
A. Lubotzky, R. Phillips, P. Sarnak, Ramanujan Graphs, \textit{Combinatorica} \textbf{8} (1988) 261--277.

\bibitem{N91}
A. Nilli, On the second eigenvalue of a graph, \textit{Discrete Math.} \textbf{91} (1991) 207--210.

\bibitem{NS94}
N. Nisan, M. Szegedy, On the degree of Boolean functions as real polynomials, \textit{Comput. Complexity} \textbf{4} (1994) 301--313.

\bibitem{S91}
J. Shao, Bounds on the $k$th eigevalues of trees, \textit{Linear Algebra Appl.} \textbf{149} (1991) 19--34.

\bibitem{S95}
J. Shao, On the largest $k$th eigenvalues of trees, \textit{Linear Algebra Appl.} \textbf{221} (1995) 131--157.

\bibitem{S87}
R. P. Stanley, A bound on the spectral radius of graphs with $e$ edges, \textit{Linear Algebra Appl.} \textbf{87} (1987) 267--269.

\bibitem{S04}
D. Stevanovic, The largest eigenvalue of nonregular graphs, \textit{J. Combin. Theory Ser. B} \textbf{91}(1) (2004) 143--146.

\bibitem{Z05}
X.-D. Zhang, Eigenvectors and eigenvalues of non-regular graphs, \textit{Linear Algebra Appl.} \textbf{409} (2005) 79--86.


\bibitem{S74}
A. J. Schwenk, Computing the Characteristic polynomial of a graph. In: Bari, R.A., Harary, F. (eds) \textit{Graphs and Combinatorics}. Lecture Notes in Mathematics, vol 406. Springer, Berlin, Heidelberg, 1974.

\bibitem{Sun+Das}
S. W. Sun, K. C. Das, A conjecture on spectral radius of graphs, \textit{Linear Algebra Appl.} \textbf{588} (2020) 74--80.

\bibitem{WG26+}
Z. Wang, J. Guo, Extremal results on the second largest eigenvalue of graphs with given order, to appear.

\end{thebibliography}


\newpage

\appendix

\setcounter{equation}{0}
\renewcommand{\theequation}{A.\arabic{equation}}

\section{Appendix: Proof of Lemma \ref{lem:four special graphs-2}}\label{appendix A}

The graph $G$ is of the form $H_1uvH_2$ with $H_1\in \mathcal{H}_{n_1}$ and $H_2\in\mathcal{H}_{n_2}$.
It suffices to prove that $S_{\frac{n}{2}}uvS_{\frac{n}{2}}$ is the unique graph $G$ satisfying $\lambda_2(G)\ge \sqrt{m(G)-\frac{n(G)}{2}-\frac{1}{2}}$.
From Lemma \ref{lem:connected by negative edges}, we get $\lambda_2(G)<\min\{\rho(H_1),\rho(H_2)\}$.
We thus establish the following relation, which will be used frequently in the sequel,
\begin{align}\label{eq:appendix-lower}
    \rho(H_1)>\sqrt{m(G)-\frac{n(G)}{2}-\frac{1}{2}}\ \  \textrm{and}\ \  \rho(H_2)>\sqrt{m(G)-\frac{n(G)}{2}-\frac{1}{2}}.
\end{align}

For understanding, let $u_i$ (or $v_i$, resp.) be a vertex of $H_1\in \mathcal{H}_{n_1}\backslash\{S^1_{n_1-1}\}$ (or $H_2\in \mathcal{H}_{n_2}\backslash\{S^1_{n_2-1}\}$, resp.) with the $i$-th smallest degree.
For example, if $H_1=S_{\frac{n_1}{2}}$, then $u_1$ is a pendent vertex of $S_{\frac{n_1}{2}}$, and $u_2$ is the center of $S_{\frac{n_1}{2}}$.
If $H_1=S^1_{\frac{n_1}{2}-1}$, then label as $u_2$ the vertex of degree $2$,  as $u_1$ the pendent vertex adjacent to $u_2$, as $u_3$ a pendant vertex except $u_1$, and as $u_4$ the vertex of degree $n_1-2$.
Label as $v_i$ a vertex in $H_2$ by a similar way if $H_2=S^1_{\frac{n_2}{2}-1}$. 
See two illustrations in Figure 3.

\vskip0.382in

\setlength{\unitlength}{0.5pt}
\begin{center}
\begin{picture}(761,160)
\put(0,160){\circle*{6}}
\put(40,80){\circle*{6}}
\qbezier(0,160)(20,120)(40,80)
\put(80,160){\circle*{6}}
\qbezier(40,80)(60,120)(80,160)
\put(120,80){\circle*{6}}
\qbezier(40,80)(80,80)(120,80)
\put(199,80){\circle*{6}}
\put(279,80){\circle*{6}}
\qbezier(199,80)(239,80)(279,80)
\put(240,160){\circle*{6}}
\qbezier(240,160)(259,120)(279,80)
\put(320,160){\circle*{6}}
\qbezier(279,80)(299,120)(320,160)
\linethickness{1.5pt}
\qbezier(120,80)(159,80)(199,80)
\put(29,160){\circle*{4}}
\put(40,160){\circle*{4}}
\put(51,160){\circle*{4}}
\put(269,160){\circle*{4}}
\put(280,160){\circle*{4}}
\put(290,160){\circle*{4}}
\put(440,160){\circle*{6}}
\put(480,80){\circle*{6}}
\thinlines
\qbezier(440,160)(460,120)(480,80)
\put(520,160){\circle*{6}}
\qbezier(520,160)(500,120)(480,80)
\put(560,80){\circle*{6}}
\qbezier(480,80)(520,80)(560,80)
\put(640,80){\circle*{6}}
\put(720,80){\circle*{6}}
\qbezier(640,80)(680,80)(720,80)
\put(469,160){\circle*{4}}
\put(480,160){\circle*{4}}
\put(491,160){\circle*{4}}
\qbezier(681,120)(681,103)(692,91)\qbezier(692,91)(704,80)(721,80)\qbezier(721,80)(737,80)(749,91)\qbezier(749,91)(761,103)(761,120)\qbezier(681,120)(681,136)(692,148)\qbezier(692,148)(704,160)(721,160)\qbezier(721,160)(737,160)(749,148)\qbezier(749,148)(761,136)(761,120)
\put(109,63){$u_1$}
\put(30,63){$u_2$}
\put(189,63){$v_1$}
\put(269,60){$v_2$}
\put(550,63){$u_1$}
\put(470,63){$u_2$}
\put(631,63){$v_1$}
\put(761,150){$v_2$}
\put(711,63){$v_3$}
\put(755,140){\circle*{6}}
\linethickness{1.5pt}
\qbezier(560,80)(600,80)(640,80)
\thinlines
\put(121,30){$S_{\frac{n}{2}}u_1v_1S_{\frac{n}{2}}$}
\put(551,31){$S_{n_1}u_1v_1K^1_{n_2-1}$}
\put(698,111){$K_{n_2-1}$}
\put(100,-20){Figure 3: Two illustrations $S_{\frac{n}{2}}u_1v_1S_{\frac{n}{2}}$ and $S_{n_1}u_1v_1K^1_{n_2-1}$}
\end{picture}
\end{center}
\vskip0.382in

Up to isomorphic, omitting the value of $n_1$ and $n_2$, there are $^\#G=55$ kinds of graphs $G$ of the form $H_1uvH_2$ (see Table \ref{tab1}).
We distinguish into 6 types.

\vskip0.2in

\begin{table}[h!]
\centering
\caption{The number of graphs $H_1u_iv_jH_2$}\label{tab1}
\begin{tabular}{|p{2.5cm}|p{2.5cm}|p{2.5cm}|p{2.5cm}|p{2.5cm}|}
\hline
\makecell[c]{$^\#H_1u_iv_jH_2$} & \makecell[c]{$S_{n_2}$\\ \small{$v_1\ \ v_2$}} &  \makecell[c]{$S^1_{n_2-1}$\\ \small{$v_1\ \ v_2\ \ v_3\ \ v_4$}} & \makecell[c]{$K_{n_2}$\\ \small{$v_1$}} & \makecell[c]{$K^1_{n_2-1}$\\ \small{$v_1\ \ v_2\ \ v_3$}}\\
\hline
\makecell[c]{$S_{n_1}$\\ \small{$u_1\ \ u_2$}} & \makecell[c]{3} & \makecell[c]{8} & \makecell[c]{2} & \makecell[c]{6}\\
\hline
\makecell[c]{$S^1_{n_1-1}$\\ \small{$u_1\ \ u_2\ \ u_3\ \ u_4$}} & \makecell[c]{--} & \makecell[c]{10} & \makecell[c]{4} & \makecell[c]{12}\\
\hline
\makecell[c]{$K_{n_1}$\\ \small{$u_1$}} & \makecell[c]{--} & \makecell[c]{--} & \makecell[c]{1} & \makecell[c]{3}\\
\hline
\makecell[c]{$K^1_{n_1-1}$\\ \small{$u_1\ \ u_2\ \ u_3$}} & \makecell[c]{--} & \makecell[c]{--} & \makecell[c]{--} & \makecell[c]{6}\\
\hline
\end{tabular}
\end{table}

\vskip0.2in

By Computing, we can get the spectral radius of graphs in $\mathcal{H}_n$.
Follow the label of graph in $\mathcal{H}_{n_1}$.
For the graph $S^1_{n-1}$, let $\textit{\textbf{y}}$ be an eigenvector of $S^1_{n-1}$ corresponding to $\rho(S^1_{n-1})$.
Clearly, $\textit{\textbf{y}}_{u_4}=\max\{\textit{\textbf{y}}_{u_1},\textit{\textbf{y}}_{u_2},\textit{\textbf{y}}_{u_3},\textit{\textbf{y}}_{u_4}\}$. 
Also $\rho(S^1_{n-1})\textit{\textbf{y}}_{u_1}=\textit{\textbf{y}}_{u_2}<\textit{\textbf{y}}_{u_4}$, that is, $\textit{\textbf{y}}_{u_1}<\frac{1}{\rho(S^1_{n-1})}\textit{\textbf{y}}_{u_4}$.
Thus,
$$\rho^2(S^1_{n-1})\textit{\textbf{y}}_{u_4}=(n-2)\textit{\textbf{y}}_{u_4}+\textit{\textbf{y}}_{u_1}<\Big(n-2+\frac{1}{\rho(S^1_{n-1})}\Big)\textit{\textbf{y}}_{u_4},$$
which follows that
$$\rho(S^1_{n-1})<\sqrt{n-2+\frac{1}{n-2}},$$
since $\rho(S^1_{n-1})>\sqrt{n-2}$.
For the graph $K^1_{n-1}$, let $\textit{\textbf{z}}$ be an eigenvector of $K^1_{n-1}$ corresponding to $\rho(K^1_{n-1})$.
Clearly, $\textit{\textbf{z}}_{u_3}=\max\{\textit{\textbf{z}}_{u_1},\textit{\textbf{z}}_{u_2},\textit{\textbf{z}}_{u_3}\}$.
Then
$$\rho(K^1_{n-1})\textit{\textbf{z}}_{u_3}=(n-2)\textit{\textbf{z}}_{u_2}+\textit{\textbf{z}}_{u_1}<(n-2)\textit{\textbf{z}}_{u_3}+\frac{1}{\rho(K^1_{n-1})}\textit{\textbf{z}}_{u_3},$$
which follows that
$$\rho(K^1_{n-1})<n-2+\frac{1}{n-2},$$
since $\rho(K^1_{n-1})>n-2$.
We also know that
\begin{center}
    $\rho(S_n)=\sqrt{n-1}$\ \ \ and\ \ \ $\rho(K_n)=n-1$.
\end{center}

From \cite[Theorem \textcolor{blue}{3}]{S74},  the characteristic polynomial of $H_1u_iv_jH_2$ is
$$\phi_{H_1u_iv_jH_2}(x)=\phi_{H_1}(x)\phi_{H_2}(x)-\phi_{H_1-u_i}(x)\phi_{H_2-v_j}(x).$$
It not hard to know $\rho(G)>\sqrt{m-\frac{n}{2}-\frac{1}{2}}$ and $\lambda_3(G)<\sqrt{m-\frac{n}{2}-\frac{1}{2}}$ by Cauchy interlacing theorem.
So $\lambda_2(H_1u_iv_jH_2)>\sqrt{m-\frac{n}{2}-\frac{1}{2}}$ if $\phi_{H_1u_iv_jH_2}(\sqrt{m-\frac{n}{2}-\frac{1}{2}})>0$, and $\lambda_2(H_1u_iv_jH_2)<\sqrt{m-\frac{n}{2}-\frac{1}{2}}$ if $\phi_{H_1u_iv_jH_2}(\sqrt{m-\frac{n}{2}-\frac{1}{2}})<0$.

Now we write $n(G)$ as $n$, and $m(G)$ as $m$ for simplification.
Then $n=n_1+n_2$ and $m=m(H_1)+m(H_2)+1$.

\vskip0.25in

\noindent\textbf{\large Type I}: $H_1\in\{S_{n_1},S^1_{n_1-1}\}$ and $H_2\in\{S_{n_2},S^1_{n_2-1}\}$.

From Eq.\,(\ref{eq:appendix-lower}), we have $\sqrt{n_1-1}\ge\rho(H_1)>\sqrt{m-\frac{n}{2}-\frac{1}{2}}=\sqrt{\frac{n-3}{2}}$, which gives $n_1>\frac{n-1}{2}$.
Similarly, we get $n_2>\frac{n-1}{2}$. So $n_1=n_2=\frac{n}{2}$.

Since $\rho(S^1_{\frac{n}{2}-1})<\sqrt{\frac{n}{2}-2+\frac{1}{\frac{n}{2}-2}}<\sqrt{\frac{n-3}{2}}$, we get $H_1=S_{\frac{n}{2}}$ and $H_2=S_{\frac{n}{2}}$.
Computing the characteristic polynomial of $G$, we have
\begin{align*}
    &\phi_{S_{\frac{n}{2}}u_1v_1S_{\frac{n}{2}}}(x)=x^{n-6}\left[x^2\left(x^2-\big(\frac{n}{2}-1\big)\right)^2-\left(x^2-\big(\frac{n}{2}-2\big)\right)^2\right],\\
    &\phi_{S_{\frac{n}{2}}u_1v_2S_{\frac{n}{2}}}(x)=x^{n-4}\left[\left(x^2-\big(\frac{n}{2}-1\big)\right)^2-\left(x^2-\big(\frac{n}{2}-2\big)\right)\right],\\
    &\phi_{S_{\frac{n}{2}}u_2v_2S_{\frac{n}{2}}}(x)=x^{n-4}\left[\left(x^2-\big(\frac{n}{2}-1\big)\right)^2-x^2\right].
\end{align*}
Then $\phi_{S_{\frac{n}{2}}u_1v_1S_{\frac{n}{2}}}(\sqrt{\frac{n-3}{2}})=\left(\frac{n-3}{2}\right)^{\frac{n-6}{2}}\cdot\frac{n-5}{8}>0$, and $\phi_{S_{\frac{n}{2}}u_1v_2S_{\frac{n}{2}}}(\sqrt{\frac{n-3}{2}})=\left(\frac{n-3}{2}\right)^{\frac{n-4}{2}}\cdot(-\frac{1}{4})<0$ and $\phi_{S_{\frac{n}{2}}u_1v_1S_{\frac{n}{2}}}(\sqrt{\frac{n-3}{2}})=\left(\frac{n-3}{2}\right)^{\frac{n-4}{2}}\cdot\frac{7-2n}{4}<0$, which indicates that
\begin{center}
    $\lambda_2(S_{\frac{n}{2}}u_1v_1S_{\frac{n}{2}})>\sqrt{\frac{n-3}{2}}$,\ \ $\lambda_2(S_{\frac{n}{2}}u_1v_2S_{\frac{n}{2}})<\sqrt{\frac{n-3}{2}}$\ \ and\ \ $\lambda_2(S_{\frac{n}{2}}u_2v_2S_{\frac{n}{2}})<\sqrt{\frac{n-3}{2}}$.
\end{center}
Thus, we obtain $G$ is $S_{\frac{n}{2}}u_1v_1S_{\frac{n}{2}}$.

\vskip0.2in

\noindent\textbf{\large Type II}: $H_1\in\{S_{n_1},S^1_{n_1-1}\}$ and $H_2=K_{n_2}$.

We have $\rho(H_1)\le\sqrt{n_1-1}$, $\rho(H_2)=n_2-1$, and
$m-\frac{n}{2}-\frac{1}{2}=n_1-1+\frac{n_2(n_2-1)}{2}+1-\frac{n_1+n_2}{2}-\frac{1}{2}=\frac{n_2^2+n_1-2n_2-1}{2}$.
From Eq.\,(\ref{eq:appendix-lower}), 
$\sqrt{n_1-1}>\sqrt{\frac{n_2^2+n_1-2n_2-1}{2}}$, which gives
$n_1>(n_2-1)^2$; and $n_2-1>\sqrt{\frac{n_2^2+n_1-2n_2-1}{2}}$, which gives $n_1<(n_2-1)^2$. This is a contradiction.
So there exists no graph for this type.

\vskip0.2in

\noindent\textbf{\large Type III}: $H_1\in\{S_{n_1},S^1_{n_1-1}\}$ and $H_2=K^1_{n_2-1}$.

Comparing the characteristic polynomials of $H_1u_iv_1H_2$ and $H_1u_iv_jH_2$ for $j=2,3$, we have
\begin{align*}
    &\phi_{H_1u_iv_1H_2}(x)-\phi_{H_1u_iv_2H_2}(x)=\phi_{H_1-u_i}(x)\cdot \left(n_2-3\right)(x+2)(x+1)^{n_2-4},\\
    &\phi_{H_1u_iv_1H_2}(x)-\phi_{H_1u_iv_3H_2}(x)=\phi_{H_1-u_i}(x)\cdot \left(n_2-2\right)(x+1)^{n_2-3}.
\end{align*}
This gives that $\phi_{H_1u_iv_1H_2}(x)>\max\{\phi_{H_1u_iv_2H_2}(x), \phi_{H_1u_iv_3H_2}(x)\}$ when $x>\rho(H_1-u_i)$.
Since $\lambda_2(H_1u_iv_jH_2)\ge \rho(H_1-u_i)$,
we get 
$$\lambda_2(H_1u_iv_1H_2)>\max\{\lambda_2(H_1u_iv_2H_2), \lambda_2(H_1u_iv_3H_2)\}.$$

If $H_1=S_{n_1}$, then 
$$\phi_{S_{n_1}u_1v_1K^1_{n_2-1}}(x)-\phi_{S_{n_1}u_2v_1K^1_{n_2-1}}(x)=(n_1-2)x^{n_1-3}\phi_{K_{n_2-1}}(x)>0,$$
implying that 
$$\lambda_2(S_{n_1}u_1v_1K^1_{n_2-1})>\lambda_2(S_{n_1}u_2v_1K^1_{n_2-1}).$$
Similarly, if $H_1=S^1_{n_1-1}$, then we can show 
$$\lambda_2(S^1_{n_1-1}u_1v_1K^1_{n_2-1})>\max\{\lambda_2(S^1_{n_1-1}u_2v_1K^1_{n_2-1}),\lambda_2(S^1_{n_1-1}u_3v_1K^1_{n_2-1})\}.$$

Now we compare the characteristic polynomials of $S_{n_1}u_1v_1K^1_{n_2-1}$ and $S^1_{n_1-1}u_1v_1K^1_{n_2-1}$.
Then 
$$\phi_{S_{n_1}u_1v_1K^1_{n_2-1}}(x)-\phi_{S^1_{n_1-1}u_1v_1K^1_{n_2-1}}(x)=-(n_1-3)x^{n_1-4}\phi_{K^1_{n_2-1}}(x)>0$$
when $x\in\left(\lambda_2(K^1_{n_2-1}),\rho(K^1_{n_2-1})\right)$.
Since $\lambda_2(S_{n_1}u_1v_1K^1_{n_2-1})\in \left(\lambda_2(K^1_{n_2-1}),\rho(K^1_{n_2-1})\right)$ and $\lambda_2(S^1_{n_1-1}u_1v_1K^1_{n_2-1})\in \left(\lambda_2(K^1_{n_2-1}),\rho(K^1_{n_2-1})\right)$, we get that
$$\lambda_2(S_{n_1}u_1v_1K^1_{n_2-1})>\lambda_2(S^1_{n_1-1}u_1v_1K^1_{n_2-1}).$$

Next, assuming that $\lambda_2(S_{n_1}u_1v_1K^1_{n_2-1})\ge\sqrt{m-\frac{n}{2}-\frac{1}{2}}$, we deduce a contradiction.
We need to estimate the spectral radius of $K^1_{n-1}$ more exact.
Note $\rho(K^1_{n-1})\textit{\textbf{z}}_{v_2}=\textit{\textbf{z}}_{v_1}+(n-3)\textit{\textbf{z}}_{v_2}$ and $\rho(K^1_{n-1})\textit{\textbf{z}}_{v_1}=\textit{\textbf{z}}_{v_3}$.
Then 
$$\rho(K^1_{n-1})\textit{\textbf{z}}_{v_3}=(n-2)\textit{\textbf{z}}_{v_2}+\textit{\textbf{z}}_{v_1}=\frac{n-2}{\rho(K^1_{n-1})-(n-3)}\textit{\textbf{z}}_{v_3}+\frac{1}{\rho(K^1_{n-1})}\textit{\textbf{z}}_{v_3}.$$
So $\rho(K^1_{n-1})$ is the largest root of $x^3 - (n-3)x^2 - (n-1)x + (n-3) = 0$.
It is checked that $\rho(K^1_{n-1})<n-2+\frac{1}{(n-2)^2}$.

We have $\rho(S_{n_1})=n_1-1$, $\rho(K^1_{n_2-1})<n_2-2+\frac{1}{(n_2-2)^2}$, 
and $m-\frac{n}{2}-\frac{1}{2}=(n_1-1)+1+\frac{(n_2-1)(n_2-2)}{2}+1-\frac{n_1+n_2}{2}-\frac{1}{2}=\frac{(n_1-1)+(n_2-2)^2}{2}$.
From Eq.\,(\ref{eq:appendix-lower}), we get
$n_1-1>\sqrt{\frac{(n_1-1)+(n_2-2)^2}{2}}$ and $n_2-2+\frac{1}{(n_2-2)^2}>\sqrt{\frac{(n_1-1)+(n_2-2)^2}{2}}$, which follows that, respectively,
\begin{align*}
    n_1-1>(n_2-2)^2,
\end{align*}
and
$$n_1-1<(n_2-2)^2+\frac{4}{n_2-2}+\frac{2}{(n_2-2)^4}.$$
Since $n_1$ and $n_2$ are integers, and $n_2\ge 10$, we get
$n_1-1\ge(n_2-2)^2+1$ and $n_1-1\le(n_2-2)^2$, a contradiction.

Thus, any graph $G$ of this type satisfies $\lambda_2(G)<\sqrt{m-\frac{n}{2}-\frac{1}{2}}$.
There exists no graph.

\vskip0.2in

\noindent\textbf{\large Type IV}: $H_1=K^1_{n_1-1}$ and $H_2=K^1_{n_2-1}$.

We have $\rho(H_1)<n_1-2+\frac{1}{n_1-2}$, $\rho(H_2)<n_2-2+\frac{1}{n_2-2}$, 
and $m-\frac{n}{2}-\frac{1}{2}=\frac{(n_1-1)(n_1-2)}{2}+1+\frac{(n_2-1)(n_2-2)}{2}+1+1-\frac{n_1+n_2}{2}-\frac{1}{2}=\frac{(n_1-1)^2+(n_2-1)^2+1}{2}\ge \frac{(n-4)^2+2}{4}$.
From Eq.\,(\ref{eq:appendix-lower}), we get
$n_1-2+\frac{1}{n_1-2}>\sqrt{\frac{(n-4)^2+2}{4}}$ and $n_2-2+\frac{1}{n_2-2}>\sqrt{\frac{(n-4)^2+2}{4}}$, which follows that
\begin{align*}
    n_1&>\sqrt{\frac{(n-4)^2-6}{4}-\frac{1}{(n_1-2)^2}}+2\\
       &>\sqrt{\frac{(n-4)^2-7}{4}}+2\\
       &>\frac{n-4}{2}-\frac{2}{n-4}+2\\
       &>\frac{n-1}{2},
\end{align*}
and $n_2>\frac{n-1}{2}$ similarly.
So $n_1=n_2=\frac{n}{2}$.

Comparing the characteristic polynomials of $H_1u_iv_1H_2$ and $H_1u_iv_jH_2$ for $j=2,3$, we have
\begin{align*}
    &\phi_{H_1u_iv_1H_2}(x)-\phi_{H_1u_iv_2H_2}(x)=\phi_{H_1-u_i}(x)\cdot \left(\frac{n}{2}-3\right)(x+2)(x+1)^{\frac{n}{2}-4},\\
    &\phi_{H_1u_iv_1H_2}(x)-\phi_{H_1u_iv_3H_2}(x)=\phi_{H_1-u_i}(x)\cdot \left(\frac{n}{2}-2\right)(x+1)^{\frac{n}{2}-3}.
\end{align*}
This gives that $\phi_{H_1u_iv_1H_2}(x)>\max\{\phi_{H_1u_iv_2H_2}(x), \phi_{H_1u_iv_3H_2}(x)\}$ when $x>\sqrt{\frac{(n-4)^2+2}{4}}$.
So 
$$\lambda_2(H_1u_iv_1H_2)>\max\{\lambda_2(H_1u_iv_2H_2), \lambda_2(H_1u_iv_3H_2)\}.$$
Similarly, by considering the index $i$, we can get that
$$\lambda_2(H_1u_1v_1H_2)>\max\{\lambda_2(H_1u_iv_jH_2): i\not=1\ \mathrm{or}\ j\not=1\}.$$

The characteristic polynomial of $K^1_{\frac{n}{2}-1}u_1v_1K^1_{\frac{n}{2}-1}$ satisfies
\begin{align*}
    \frac{\phi_{K^1_{\frac{n}{2}-1}u_1v_1K^1_{\frac{n}{2}-1}}(x)}{(x+1)^{n-6}}&=\left[ x^3 - \left(\frac{n}{2}-2\right)x^2 - 2x + (n-5) \right]\cdot \left[ x^3 - \left(\frac{n}{2}-4\right)x^2 - (n-4)x - 1 \right]\\
    &\triangleq f_1(x)\cdot f_2(x).
\end{align*}
It is checked that $f_1(x)$ and $f_2(x)$ are strictly increasing on $x>\sqrt{\frac{(n-2)^2+2}{4}}$, and 
$$f_1(x)\cdot f_2(x)\Big|_{x=\sqrt{\frac{a^2+2}{4}}}=\frac{a^2-4a+2}{4\sqrt{a^2+2}+a}\left(\frac{a^2-6}{4\sqrt{a^2+2}+a}-1\right)>0,$$
where $a=n-4$.
This infers that $\lambda_2(K^1_{\frac{n}{2}-1}u_1v_1K^1_{\frac{n}{2}-1})<\rho(K^1_{\frac{n}{2}-1}u_1v_1K^1_{\frac{n}{2}-1})<\sqrt{\frac{a^2+2}{4}}$, a contradiction.
So there exists no graph for this type.

\vskip0.2in

\noindent\textbf{\large Type V}: $H_1=K_{n_1}$ and $H_2=K^1_{n_2-1}$.

We have $\rho(H_1)=n_1-1$, $\rho(H_2)<n_2-2+\frac{1}{n_2-2}$, 
and $m-\frac{n}{2}-\frac{1}{2}=\frac{n_1(n_1-1)}{2}+\frac{(n_2-1)(n_2-2)}{2}+1+1-\frac{n_1+n_2}{2}-\frac{1}{2}=\frac{(n_1-1)^2+(n_2-2)^2-2}{2}\ge \frac{(n-3)^2}{4}$.
From Eq.\,(\ref{eq:appendix-lower}), we get
$n_1-1>\frac{n-3}{2}$ and $n_2-2+\frac{1}{n_2-2}>\frac{n-3}{2}$, which follows that
\begin{align*}
 n_1>\frac{n-1}{2}\ \ \mathrm{and}\ \ n_2>\frac{n+1}{2}-\frac{2}{n_2-2}.
\end{align*}
Since $n_1$ is an integer, we have $n_1\ge\frac{n}{2}$, and then $n=n_1+n_2>\frac{n}{2}+\frac{n+1}{2}-\frac{2}{n_2-2}=n+\frac{1}{2}-\frac{2}{n_2-2}$.
Then $n_2<6$. This is a contradiction.
So there exists no graph for this type.

\vskip0.2in

\noindent\textbf{\large Type VI}: $H_1=K_{n_1}$ and $H_2=K_{n_2}$.

We have $\rho(H_1)=n_1-1$, $\rho(H_2)=n_2-1$, 
and $m-\frac{n}{2}-\frac{1}{2}=\frac{n_1(n_1-1)}{2}+\frac{n_2(n_2-1)}{2}+1-\frac{n_1+n_2}{2}-\frac{1}{2}=\frac{(n_1-1)^2+(n_2-2)^2-1}{2}\ge \frac{(n-2)^2}{4}-\frac{1}{2}>\left(\frac{n-2}{2}-\frac{1}{n-2}\right)^2$.
From Eq.\,(\ref{eq:appendix-lower}), we get
$n_1-1>\frac{n-2}{2}-\frac{1}{n-2}$ and $n_2-1>\frac{n-2}{2}-\frac{1}{n-2}$, which follows that
\begin{align*}
 n_1>\frac{n}{2}-\frac{1}{n-2}\ \ \mathrm{and}\ \ n_2>\frac{n}{2}-\frac{1}{n-2}.
\end{align*}
So $n_1=n_2=\frac{n}{2}$. 
The characteristic polynomial of $K_{\frac{n}{2}}u_1v_1K_{\frac{n}{2}}$ satisfies
$$
\frac{\phi_{K_{\frac{n}{2}}u_1v_1K_{\frac{n}{2}}}(x)}{(x+1)^{n-4}} =  \left( x^2 - \left(\frac{n}{2}-1\right)x - 1 \right)\cdot \left( x^2 - \left(\frac{n}{2}-3\right)x - (n-3) \right)\triangleq f_1(x)\cdot f_2(x).
$$
It is checked that
$$\frac{\phi_{K_{\frac{n}{2}}u_1v_1K_{\frac{n}{2}}}(x)}{(x+1)^{n-4}}\Big|_{x=\sqrt{\frac{(n-2)^2}{4}-\frac{1}{2}}}<0$$
since $f_1(\sqrt{\frac{(n-2)^2}{4}-\frac{1}{2}})<0$ and $f_2(\sqrt{\frac{(n-2)^2}{4}-\frac{1}{2}})>0$.
So $\lambda_2(K_{\frac{n}{2}}u_1v_1K_{\frac{n}{2}})<\sqrt{\frac{(n-2)^2}{4}-\frac{1}{2}}$.
This is a contradiction.
So there exists no graph for this type.

\end{document}